\documentclass[12pt]{article}
\usepackage[bookmarks=false, backref=page]{hyperref}
\usepackage[left=2.00cm, right=2.00cm, top=2.50cm, bottom=2.50cm]{geometry}
\usepackage{setspace}
\usepackage{abstract}
\usepackage{cite}
\usepackage{amsthm}
\usepackage{amsmath}
\usepackage{amsfonts}
\usepackage{amssymb}
\usepackage{mathrsfs}
\usepackage{color}
\usepackage{graphicx}
\usepackage{multirow}
\usepackage{booktabs}
\usepackage{indentfirst}
\usepackage{arydshln}
\usepackage{hyperref}
\allowdisplaybreaks[4]
\numberwithin{equation}{section}
\usepackage{authblk}
\makeatletter

\newcommand{\Rmnum}[1]{\expandafter\@slowromancap\romannumeral #1@}
\makeatother

\newtheorem{theorem}{Theorem}[section]

\newtheorem{lemma}[theorem]{Lemma}

\newtheorem{proposition}[theorem]{Proposition}
\newtheorem{remark}{Remark}

    \newcommand{\BQ}{{\mathbb {Q}}} \newcommand{\BR}{{\mathbb {R}}}

     \newcommand{\BZ}{{\mathbb {Z}}}

    \newcommand{\rank}{{\mathrm{rank}}}
\newcommand{\sgn}{{\mathrm{sgn}}}

    \DeclareFontFamily{U}{wncy}{}
\DeclareFontShape{U}{wncy}{m}{n}{<->wncyr10}{}
\DeclareSymbolFont{mcy}{U}{wncy}{m}{n}
\DeclareMathSymbol{\Sha}{\mathord}{mcy}{"58}

\author{Tong Wei and Shuai Zhai}
\title{On the coefficients of the Taylor expansion of $L$-functions of elliptic curves}

\begin{document}
\date{}
\maketitle

\vspace{-1cm}

\begin{abstract}
In this paper, we investigate the coefficients of the Taylor expansion of the complex $L$-series of any elliptic curve over $\BQ$. We prove that, in the family of quadratic twists by all the discriminants $d$, these coefficients are nonvanishing under GRH when $d$ is sufficiently large. Unconditionally, we obtain a general lower bound for the number of nonvanishing coefficients in the family of quadratic twists, through a series of results from the moments of the central values of the derivatives of quadratic twists of modular $L$-function.
\end{abstract}

\section{Introduction}

Let $E$ be an elliptic curve defined over $\BQ$ with conductor $q$. Let $L(E, s)$ be the complex $L$-series of $E$. The $L$-function $L(E, s)$ is known to possess a holomorphic continuation to the whole complex plane. The CM case follows from the work of Hecke and Tate, while the general case for elliptic curves over $\mathbb{Q}$ was established by Wiles \cite{Wiles}, Taylor--Wiles \cite{Taylor}, and Breuil--Conrad--Diamond--Taylor \cite{Breuil}. Thus, there is a Taylor expansion of $L(E, s)$ at $s=1$ with the form
\begin{equation}\label{Taylor}
L(E, s)=\sum_{i=r}^{\infty} c_i (s-1)^i=c_r(s-1)^r+c_{r+1}(s-1)^{r+1} + \cdots,
\end{equation}
where $r$ is a non-negative integer which is defined to be the analytic rank of $E$. The Birch and Swinnerton-Dyer conjecture asserts that $r=\rank(E(\BQ))$. Moreover, the conjecture predicts a deep connection between $c_r$ and the arithmetic invariants of $E$. For recent progress on this conjecture, we refer the reader to a nice survey article by Burungale--Skinner--Tian \cite{BST}.

Thanks to the breakthrough of Yun--Zhang \cite{YZ1,YZ2}, an extension of the Gross--Zagier and Waldspurger formulae for higher derivatives of more general complete $L$-functions over function fields are given. This development provides a geometric explanation for all the coefficients appearing in the Taylor expansion of these $L$-functions over function fields. However, over number fields, there is still no any arithmetic explanation for the other coefficients of \eqref{Taylor}, say $c_{r+j}$ for integers $j \geq 1$. In the current paper, we investigate the coefficients of \eqref{Taylor} in the family of quadratic twists through an analytic view. For each discriminant $d$ of a quadratic extension of $\BQ$, we write $E^{(d)}$ for the twist of $E$ by this quadratic extension. As usual, we have the following Taylor expansion of $L(E^{(d)}, s)$ at $s=1$ with the form
\begin{equation}\label{Taylor2}
L(E^{(d)}, s)=\sum_{i=r_d}^{\infty} c^{(d)}_i (s-1)^i=c^{(d)}_{r_d}(s-1)^{r_d}+c^{(d)}_{r_d+1}(s-1)^{r_d+1} + \cdots,
\end{equation}
where $r_d$ is the analytic rank of $E^{(d)}$. When $d$ is varying, $r_d$ could be $0, 1, 2, \cdots$. For example, when the root number of the functional equation of $L(E^{(d)},s)$ is $+1$ (or $-1$), Goldfeld \cite{G} proposed a nice conjecture on how often $r_d$ is $0$, and $1$, i.e., the behaviour of non-vanishing of $c^{(d)}_0$ (or $c^{(d)}_1$). It is conjectured that, for every elliptic curve defined over $\BQ$, amongst the set of all its quadratic twists with root number $+1$, there is a subset of density one where the central $L$-value of the twist is non-zero, and amongst the set of all its quadratic twists with root number $-1$, there is a subset of density one where the central $L$-value of the twist has a zero of order equal to 1. For any elliptic curve $E$ over $\BQ$, we prove the following result.

\begin{theorem}\label{thm3}
Assuming the Riemann Hypothesis for $L(E^{(d)}, s)$, for any positive integer $h$, if $\lvert d \rvert \geq \frac{2\pi}{\sqrt{q}} e^{h^3 e^{\gamma+1}}$, then
$$
c_{m}^{(d)} \neq 0
$$
for all $r_d \leq m \leq h$, where $\gamma$ denotes the Euler constant. Furthermore, the sign of $c_{m}^{(d)}$ is given by $\omega_{E^{(d)}}(-1)^m$, where $\omega_{E^{(d)}}$ is the root number of $E^{(d)}$.
\end{theorem}

\begin{remark}
As we have mentioned, $r_d=0, 1$ for $100\%$ quadratic twists is predicted by Goldfeld's conjecture. Recently, Goldfeld's conjecture has been proved by a series of remarkable work of Smith \cite{Smith1}\cite{Smith2}\cite{Smith3} by assuming the Birch--Swinnerton-Dyer conjecture. Unconditionally, by the celebrated work of Burungale--Tian \cite{BT}, Goldfeld conjecture for the case $r_d=0$ has been proved for the congruent number elliptic curve.
\end{remark}
\begin{remark}
It is worth noting that for relatively small values of $d$, the sign of $c_{m}^{(d)}$ may not necessarily satisfy the relation $\omega_{E^{(d)}}(-1)^m$. Nevertheless, we conjecture that $c_{m}^{(d)}$ is always nonvanishing for all $d$. As an illustration, we provide two examples for the congruent number elliptic curve $E: y^2=x^3-x$.
\begin{itemize}
\item $E^{(5)}: y^2=x^3-5^2x$, $\omega_{E^{(5)}}=-1$, $r_5=1$, the Taylor expansion of $L(E^{(5)}, s)$ at $s=1$ is
\begin{fontsize}{9pt}{2pt}
\begin{align*}
&0.00 + 2.23*(s-1) - 2.07*(s-1)^2
+ 0.55*(s-1)^3+ 0.97*(s-1)^4 - 1.57*(s-1)^5 + 1.32*(s-1)^6\\
& - 0.75*(s-1)^7 + 0.29*(s-1)^8
- 0.05*(s-1)^9 - 0.03*(s-1)^{10}+ 0.04*(s-1)^{11}+ O((s-1)^{12}).
\end{align*}
\end{fontsize}
\item $E^{(193)}: y^2=x^3-193^2x$, $\omega_{E^{(193)}}=1$, $r_{193}=0$, the Taylor expansion of $L(E^{(193)}, s)$ at $s=1$ is
\begin{fontsize}{9pt}{2pt}
\begin{align*}
&0.75 - 3.46*(s-1) + 28.45*(s-1)^2
- 105.81*(s-1)^3 + 262.52*(s-1)^4 - 488.96*(s-1)^5 + 718.88*(s-1)^6 \\
& - 849.12*(s-1)^7 + 794.71*(s-1)^8
- 541.34*(s-1)^9 + 157.44*(s-1)^{10} + 241.02*(s-1)^{11} +O((s-1)^{12}).
\end{align*}
\end{fontsize}
\end{itemize}
\end{remark}

We will prove Theorem \ref{thm3} in Section 2, by applying the super-positivity of $L$-values of Yun--Zhang \cite{YZ1}. Unconditionally, by applying a series of results on moments of modular $L$-functions we prove the following nonvanishing results for $c^{(d)}_i$.

\begin{theorem}\label{thm1}
We have
$$\#\{|d| \leq X | c^{(d)}_i \neq 0 \} \gg \frac{X}{\log X}$$
for any integer $i \geq 0$.
\end{theorem}

\begin{theorem}\label{thm2}
There are infinitely many $d$ such that $c^{(d)}_i c^{(d)}_j \neq 0$ for any fixed $i, j \in \BZ_{\geq 0}$.
\end{theorem}

The assertion for $c^{(d)}_0$ in Theorem \ref{thm1} has been obtained by Ono--Skinner \cite{OS} by applying Waldspurger formula. On the other hand, if we apply the remarkable result on the second moment of modular $L$-functions of Li \cite{Li}, combining with the recent work given by Shen \cite{SQ}, to arbitrary levels, we can also get the same lower bound $X/\log X$. For $c^{(d)}_1$, in the family of root number $-1$, Kumar--Mallesham--Sharma--Singh\cite{KMSS} got the lower bound $X/\log X$ by combining the work of Murty--Murty\cite{MM} and Iwaniec\cite{I}. We will prove the above two results by applying a series of results about the moments of the central values of the derivatives of quadratic twists of modular $L$-functions. We will prove Theorem \ref{thm1} and Theorem \ref{thm2} in Section 4. The proof is based on the recent work of Li \cite{Li}, Shen \cite{SQ}, Sono \cite{Son}, and Soundararajan--Young\cite{SY}.

\medskip

{\bf Acknowledgements}. The authors would like to thank Dorian Goldfeld, Chao Li, Xiannan Li, Quanli Shen, Kannan Soundararajan, and Wei Zhang for their helpful discussions and suggestions. We are also grateful to Jianya Liu and Andrew Wiles for their encouragement, and to Zeev Rudnick for numerous insightful suggestions that significantly improved the structure and content of this paper. Shuai Zhai would like to thank Barry Mazur for his hospitality and the key insights shared during a visit to Harvard while this work was in progress.

\section{Preliminary results}


Let $E/\mathbb{Q}$ be any elliptic curve of conductor $q$. There exists a primitive cusp form with level $q$ and weight $2$, say
\begin{equation*}
f(z)=\sum_{n=1}^{\infty}\lambda_{f}(n)n^{\frac{1}{2}}e(nz)\in S_{2}(\Gamma_0(q)),
\end{equation*}
such that
\begin{equation*}
L(E, s+1/2)=L(s, f)=\sum_{n=1}^{\infty}\lambda_{f}(n)n^{-s}
\end{equation*}
with $\lambda_f(1)=1$ and $f$ has been normalized so that the Deligne's bound gives $|\lambda_f(n)|\leq\tau(n)$ for all $n\geq1$, where $\tau(n)$ is the divisor function. We study the family of twists of $f$ by quadratic characters. Let $d$ denote a fundamental discriminant, and $\chi_d(\cdot)=\left(\frac{d}{\cdot}\right)$ denote the primitive quadratic character of conductor $\lvert d\rvert$. Then $f\otimes\chi_d$ is a primitive Hecke eigenform of level $q\lvert d\rvert^{2}$ and the twisted $L$-function is defined for $\Re(s)>1$ by
\begin{equation*}
L(s, f \otimes \chi_d)=\sum_{n=1}^{\infty}\frac{\lambda_f(n)\chi_d(n)}{n^s}=\prod_{p\nmid qd}\left(1-\frac{\lambda_f(p)\chi_{d}(p)}{p^s}+\frac{1}{p^{2s}}\right)^{-1}
\prod_{p\mid q}\left(1-\frac{\lambda_f(p)\chi_{d}(p)}{p^s}\right)^{-1}.
\end{equation*}
The completed $L$-function is defined by
\begin{equation}
\Lambda\left(s, f \otimes \chi_d\right)=\left(\frac{\sqrt{q}\lvert d\rvert}{2 \pi}\right)^s \Gamma\left(s+1/2\right) L(s, f \otimes \chi_d).
\end{equation}
It has the functional equation
\begin{equation}
\Lambda\left(s, f \otimes \chi_d\right)=-\eta\chi_d(-q)\Lambda\left(1-s, f \otimes \chi_d\right),
\end{equation}
where $\eta$ is the eigenvalue of the Fricke involution which is independent of $q$ and always $\pm1$. We denote the root number by $\omega(f \otimes \chi_d):=-\eta\chi_d(-q)$. Note that if $d$ is a fundamental discriminant, then $\chi_d(-1)=\pm1$ depending on whether $d$ is positive or negative.

We are interested in the $l$-th derivative of the $L$-function, which also has a Dirichlet series convergent in a right half-plane
\begin{equation*}
L^{(l)}(s, f\otimes\chi_d)=(-1)^{l}\sum_{n\geq1}\frac{\lambda_f(n)\chi_d(n)\log^ln}{n^s}.
\end{equation*}
It also has a functional equation
\begin{equation*}
\Lambda^{(l)}(s, f\otimes\chi_d)=(-1)^{l+1}\eta\chi_d(-q)\Lambda^{(l)}(1-s, f\otimes\chi_d).
\end{equation*}
Here, we first introduce the following fundamental tools.
\begin{lemma}\label{lem2.2}
Let
\begin{equation}\label{G_k}
G_k(n)=\left(\frac{1-i}{2}+\left(\frac{-1}{n}\right)\frac{1+i}{2}\right)\sum_{a \bmod n}\left(\frac{a}{n}\right)e\left(\frac{a k}{n}\right).
\end{equation}
The sum $G_k(n)$ appeared in the work of Soundararajan\cite{S}. For $m, n$ relatively prime odd integers, $G_k(mn)=G_k(m)G_k(n)$, and for $p^\alpha \| k$ (set $\alpha=\infty$ for $k=0$), then
\begin{equation*}
G_k\left(p^\beta\right)=
\begin{cases}
0, & \text { if } \beta \leq \alpha \text { is odd, } \\
\varphi\left(p^\beta\right), & \text { if } \beta \leq \alpha \text { is even, } \\
-p^\alpha, & \text { if } \beta=\alpha+1 \text { is even, } \\
\left(\frac{k p^{-\alpha}}{p}\right) p^\alpha \sqrt{p}, & \text { if } \beta=\alpha+1 \text { is odd, } \\
0, & \text { if } \beta \geq \alpha+2 .
\end{cases}
\end{equation*}
Note that $G_0(n)=\varphi(n)$ when $n$ is a perfect square, and vanishes otherwise.
\end{lemma}
\begin{lemma}[Poisson summation formula]\label{lem2.3}
Let $F$ be a Schwartz class function over the real numbers and suppose that $n$ is an odd integer. Then
\begin{equation}
\sum_d\left(\frac{d}{n}\right)F\left(\frac{d}{Z}\right)=\frac{Z}{n}\sum_{k\in\mathbb{Z}} G_k(n) \check{F}\left(\frac{kZ}{n}\right),
\end{equation}
and
\begin{equation}
\sum_{(d, 2)=1}\left(\frac{d}{n}\right)F\left(\frac{d}{Z}\right)=\frac{Z}{2n}\left(\frac{2}{n}\right)\sum_{k\in\mathbb{Z}}(-1)^k G_k(n)\check{F}\left(\frac{kZ}{2n}\right)
\end{equation}
where $G_k(n)$ is defined as in (\ref{G_k}), and the Fourier-type transform of $F$ is defined to be
\begin{equation}
\check{F}(y)=\int_{-\infty}^{\infty}(\cos (2\pi xy)+\sin (2\pi xy))F(x)dx.
\end{equation}
\end{lemma}
See Lemma 2.3 in \cite{Li} or Lemma 2.6 in \cite{S}.
We apply the following two lemmas several times in subsequent proofs. The proofs refer to Proposition 3.2, and Lemma 6.3 in \cite{Li} respectively. The function $G$ in the following lemmas is as shown in \eqref{eq.G}. In the rest of the paper, $(\frac{m}{n})$ denote the usual Kronecker symbol.
\begin{lemma}\label{lem.SMNt}
Let ${\sum}^\flat$ denotes a sum over fundamental discriminants. For $M, N\geq1$, $t\in\mathbb{R}$, we have the following estimate for the dyadic sum,
\begin{equation*}
\begin{aligned}
S^{\flat}(M, N, t)&={\sum_{M \leq|m|<2 M}}^\flat\left|\sum_n \frac{\lambda_f(n)}{n^{1/2+i t}} G\left(\frac{n}{N}\right)\left(\frac{m}{n}\right)\right|^2\\
&\ll(1+|t|)^2(M+N\log(2+N/M)),
\end{aligned}
\end{equation*}
the implied constant only depends on $f$.
\end{lemma}
\begin{lemma}\label{lem.6.3}
For any real $\mathcal{X}\geq1$, $N\geq1$, $t\in\mathbb{R}$, and positive integer $a$, we have
\begin{equation*}
\sum_{\substack{(d, 2)=1\\d\leq\mathcal{X}}}\left|\sum_{(n, a)=1}\frac{\lambda_f(n)}{n^{1/2+it}}\left(\frac{8d}{n}\right)G\left(\frac{n}{N}\right)\right|^2\ll \tau(a)^5\mathcal{X}(1+|t|)^3\log(2+|t|),
\end{equation*}
where $\tau(\cdot)$ is usually divisor function, the implied constant only depends on $f$.
\end{lemma}

The following result is due to Yun--Zhang \cite[Proposition B.1]{YZ1}.
\begin{lemma}[Yun--Zhang]\label{lem.Yun-Zhang}
Let $\phi(s)$ be an entire function with the following properties:
\begin{itemize}
    \item[(1)] it has a functional equation $\phi(-s) = \pm\phi(s)$;
    \item[(2)] for $s \in \BR $ such that $s \gg 0$, we have $\phi(s) \in \mathbb{R}_{>0}$;
    \item[(3)] the order $\rho(\phi)$ of $\phi(s)$ is at most $1$;
    \item[(4)] (RH) the only zeros of $\phi(s)$ lie on the imaginary axis $\Re(s) = 0$.
\end{itemize}
Then for all $l \geq 0$, we have
\begin{equation*}
\phi^{(l)}(0) := \left. \frac{d^l}{ds^l} \right|_{s=0} \phi(s) \geq 0.
\end{equation*}
Moreover, if $\phi(s)$ is not a constant function, we have
\begin{equation*}
\phi^{(l_0)}(0) \neq 0 \implies \phi^{(l_0+2i)}(0) \neq 0 \quad \text{for all } l_0 \text{ and } i \in \mathbb{Z}_{\geq 0}.
\end{equation*}
\end{lemma}
Now we can give the proof of Theorem \ref{thm3}.
\begin{proof}
The completed $L$-function of $E$ is defined by
\begin{equation}
\Lambda\left(E^{(d)}, s\right)=\left(\frac{\sqrt{q}\lvert d\rvert}{2 \pi}\right)^s \Gamma(s)L(E^{(d)}, s).
\end{equation}
According to Leibniz formula, for any $m\geq0$, we have
\begin{equation*}
\begin{aligned}
L^{(m)}(E^{(d)}, s)&=\sum_{k=0}^{m}\left(\begin{matrix}m\\k\end{matrix}\right)\Lambda^{(k)}(E^{(d)}, s)\\
&\quad\times\sum_{t=0}^{m-k}\left(\begin{matrix}m-k\\t\end{matrix}\right)(-1)^{m-k-t}\left(\log\frac{\sqrt{q}\lvert d\rvert}{2\pi}\right)^{m-k-t}
\left(\frac{\sqrt{q}\lvert d\rvert}{2\pi}\right)^{-s}\frac{d^{t}}{ds^{t}}\Gamma(s)^{-1}.
\end{aligned}
\end{equation*}
It can be easily deduced from Fa\`{a} di Bruno's formula\cite{J} (take $g(s)=e^{s}$, $f(t)=-\log\Gamma(t)$) that
\begin{equation*}
(\Gamma(s)^{-1})^{(t)}|_{s=1}=\sum_{\substack{k_1, k_2,\cdots,k_t\geq0\\ k_1+2k_2+\cdots+tk_t=t}}
\frac{(-1)^{k_1+k_2+\cdots+k_t+t}t!}{k_1!k_2!\cdots k_t!}\frac{\gamma^{k_1}\zeta(2)^{k_2}\zeta(3)^{k_3}\cdots\zeta(t)^{k_t}}{2^{k_2}3^{k_3}\cdots t^{k_t}},\quad t\geq1,
\end{equation*}
where $\gamma$ is Euler's constant, then
\begin{equation*}
\begin{aligned}
L^{(m)}(E^{(d)}, 1)&=\sum_{k=0}^{m}\Lambda^{(k)}(E^{(d)}, 1)\sum_{t=1}^{m-k}\frac{m!}{k!(m-k-t)!}
(-1)^{m-k-t}\left(\log\frac{\sqrt{q}\lvert d\rvert}{2\pi}\right)^{m-k-t}\left(\frac{\sqrt{q}\lvert d\rvert}{2\pi}\right)^{-1}\\
&\quad\times\sum_{\substack{k_1, k_2,\cdots,k_{t}\geq0\\ k_1+2k_2+\cdots+tk_{t}=t}}\frac{(-1)^{k_1+k_2+\cdots+k_t+t}}{k_1!k_2!\cdots k_{t}!}\frac{\gamma^{k_1}\zeta(2)^{k_2}\zeta(3)^{k_3}\cdots\zeta(t)^{k_{t}}}{2^{k_2}3^{k_3}\cdots t^{k_t}}\\
&+\sum_{k=0}^{m}\Lambda^{(k)}(E^{(d)}, 1)\frac{m!}{k!(m-k)!}(-1)^{m-k}\left(\log\frac{\sqrt{q}\lvert d\rvert}{2\pi}\right)^{m-k}\left(\frac{\sqrt{q}\lvert d\rvert}{2\pi}\right)^{-1}.
\end{aligned}
\end{equation*}
Since $c_{r_d}^{(d)}$ is the first non-zero coefficient of $L(E^{(d)}, s)$ and $s=1$ is not the pole of $\Gamma^{(n)}(s)$ for each $n\in\mathbb{N}$, we have
\begin{equation*}
\Lambda(E^{(d)}, 1)=\Lambda'(E^{(d)}, 1)=\cdots=\Lambda^{({r_d}-1)}(E^{(d)}, 1)=0,
\end{equation*}
and
\begin{equation*}
\Lambda^{({r_d})}(E^{(d)}, 1)=\left(\frac{\sqrt{q}\lvert d\rvert}{2 \pi}\right)L^{(r_d)}\left(E^{(d)}, 1\right)\neq0.
\end{equation*}
Applying Lemma \ref{lem.Yun-Zhang}, we have
\begin{equation}\label{eq.3.1}
\Lambda^{({r_d}+2j)}(E^{(d)}, 1)>0,\quad j\geq0.
\end{equation}
Hence, for $m\geq {r_d}$, it follows that
\begin{equation}\label{eq.1.3}
\begin{aligned}
L^{(m)}(E^{(d)}, 1)&=\sum_{k={r_d}}^{m}\Lambda^{(k)}(E^{(d)}, 1)\sum_{t=1}^{m-k}\frac{m!}{k!(m-k-t)!}
(-1)^{m-k-t}\left(\log\frac{\sqrt{q}\lvert d\rvert}{2\pi}\right)^{m-k-t}\\
&\quad\times\left(\frac{\sqrt{q}\lvert d\rvert}{2\pi}\right)^{-1}\sum_{\substack{k_1, k_2,\cdots,k_{t}\geq0\\ k_1+2k_2+\cdots+tk_{t}=t}}\frac{(-1)^{k_1+k_2+\cdots+k_t+t}}{k_1!k_2!\cdots k_{t}!}\frac{\gamma^{k_1}\zeta(2)^{k_2}\zeta(3)^{k_3}\cdots\zeta(t)^{k_{t}}}{2^{k_2}3^{k_3}\cdots t^{k_t}}\\
&+\sum_{k=r_d}^{m}\Lambda^{(k)}(E^{(d)}, 1)\frac{m!}{k!(m-k)!}(-1)^{m-k}\left(\log\frac{\sqrt{q}\lvert d\rvert}{2\pi}\right)^{m-k}\left(\frac{\sqrt{q}\lvert d\rvert}{2\pi}\right)^{-1}.
\end{aligned}
\end{equation}
If the root number of $E^{(d)}$ is $+1$, then all the odd-order derivatives of $\Lambda(E^{(d)}, s)$ are $0$ at the central value. Therefore, in \eqref{eq.1.3}, there are only terms with even-order derivatives (Similarly, when the root number of $E^{(d)}$ is $-1$, there are only terms with odd-order derivatives). Considering the coefficient of $\Lambda^{(2k)}(E^{(d)}, 1)$, i.e.,
\begin{equation}\label{eq.1.7}
\begin{aligned}
&\frac{m!}{(2k)!(m-2k)!}(-1)^{m-2k}\left(\frac{\sqrt{q}\lvert d\rvert}{2\pi}\right)^{-1}\left\{\left(\log\frac{\sqrt{q}\lvert d\rvert}{2\pi}\right)^{m-2k}+
\sum_{t=1}^{m-2k}\left(\log\frac{\sqrt{q}\lvert d\rvert}{2\pi}\right)^{m-2k-t}\right.\\
&\left.\times\frac{(m-2k)!}{(m-2k-t)!}\sum_{\substack{k_1, k_2,\cdots,k_{t}\geq0\\ k_1+2k_2+\cdots+tk_{t}=t}}
\frac{(-1)^{k_1+k_2+\cdots+k_t}}{k_1!k_2!\cdots k_{t}!}\frac{\gamma^{k_1}\zeta(2)^{k_2}\zeta(3)^{k_3}\cdots\zeta(t)^{k_{t}}}{2^{k_2}3^{k_3}\cdots t^{k_t}}\right\},
\end{aligned}
\end{equation}
based on the series expansion of $e^x$, it is obvious that
\begin{equation*}
\begin{aligned}
\left|\sum_{\substack{k_1, k_2,\cdots,k_{t}\geq0\\ k_1+2k_2+\cdots+tk_{t}=t}}
\frac{(-1)^{k_1+k_2+\cdots+k_t}}{k_1!k_2!\cdots k_{t}!}\frac{\gamma^{k_1}\zeta(2)^{k_2}\zeta(3)^{k_3}\cdots\zeta(t)^{k_{t}}}{2^{k_2}3^{k_3}\cdots t^{k_t}}\right|
&\leq\sum_{k_1=0}^{\infty}\frac{\gamma^{k_1}}{k_1!}\prod_{i=2}^{t}\sum_{k_i=0}^{\infty}\frac{|\zeta(i)|^{k_i}}{i^{k_i}k_i!}\\
&=e^{\gamma}e^{\sum_{i=2}^{t}\frac{|\zeta(i)|}{i}},
\end{aligned}
\end{equation*}
due to the fact that $\frac{\sigma-1}{\sigma}\leq|\zeta(s)|\leq\frac{\sigma}{\sigma-1}$ for $\Re s=\sigma>1$ and the Abel summation formula, then there is
\begin{equation}\label{eq.1.6}
e^{\gamma}e^{\sum_{i=2}^{t}\frac{|\zeta(i)|}{i}}\leq e^{\gamma}e^{\sum_{i=1}^{t}\frac{1}{i}}\leq e^{\gamma+1}t.
\end{equation}
When $\lvert d \rvert \geq\frac{2\pi}{\sqrt{q}}e^{h^3e^{\gamma+1}}$ for some $h\in\BZ^{+}$ and $r_d\leq m\leq h$, it can be easily concluded from \eqref{eq.1.6} that the expression within the braces in \eqref{eq.1.7} is greater than zero. Thus, the sign of the coefficient depends on $(-1)^{m-2k}$.
For all $k\in\mathbb{N}$, the sign of $(-1)^{m-2k}$ is only related to $m$, then the coefficient of $\Lambda^{(2k)}(E^{(d)}, 1)$ is either positive or negative simultaneously. In view of \eqref{eq.3.1}, it can be concluded that $L^{(m)}(E^{(d)}, 1)\neq0$, and its sign is $(-1)^{m}$. (When the root number is $-1$, the same result can be obtained.)  The proof of Theorem \ref{thm3} is completed.
\end{proof}

\section{A result on shifted second moment of quadratic twists of modular $L$-function}

Let $\sum^{*}$ denote a sum over squarefree integers. In this context, it is of high interest to understand moments of the form
\begin{equation}
M(l):={\sum_{\substack{0<8d<X\\(d,2)=1}}}^{*}L(1/2, f\otimes\chi_{8d})^{l},\;M'(l):={\sum_{\substack{0<8d<X\\(d,2)=1}}}^{*}L'(1/2, f\otimes\chi_{8d})^{l}.
\end{equation}
Based on random matrix theory, Keating and Snaith\cite{K} conjectured that
\begin{equation}\label{eq.1.1}
M(l)\sim C(l,f)X(\log X)^{\frac{l(l-1)}{2}},\;M'(l)\sim C'(l,f)X(\log X)^{\frac{l(l+1)}{2}}
\end{equation}

For $l=1$, Shen \cite{SQ} established the conjectured asymptotics for a smoothed version of $M(1)$ and $M'(1)$ with an error term $O(X^{\frac{1}{2}+\varepsilon})$, following the recursive method of Heath--Brown \cite{HB} and Young \cite{YM1,YM2}. Other related results on the first moment include \cite{B,I1,Luo,Mu1,Mu2,MM,P,RS,ST}. For the second moment, Soundararajan \cite{S1} showed $M(2)\ll X(\log X)^{1+\varepsilon}$ under GRH, and Soundararajan and Young \cite{SY} later proved the conjectured asymptotic for $M(2)$ conditionally. This was subsequently made unconditional by Li \cite{Li} with an error term $O(X(\log X)^{\frac{1}{2}+\varepsilon})$, which can be further improved to $O(X(\log\log X)^4)$. This significant development has provided a key technique for studying the second moment of modular $L$-functions and their derivatives. For further exploration of related problems, the reader is referred to references \cite{H,KMSS,SQ2,SQ3,Zh}.

We extend Li's\cite{Li} result to the shifted second moment of the quadratic twists of modular $L$-functions. The following theorem plays a crucial role in the proof of the main theorems.
\begin{theorem}\label{thm3.1}
Let $J(x): (0, \infty)\rightarrow\mathbb{C}$ be a smooth function compactly supported in $[1/2, 2]$. Let $U=X/(\log X)^A$, where $A$ is a sufficiently large constant. Assume $\lvert\alpha\rvert, \lvert\beta\rvert\leq\frac{1}{\log X}$, then for any normalized primitive form $f$ with odd level $q$ and weight $2$, we have
\begin{equation}\label{eq.LL}
\begin{aligned}
&\quad{\sum_{(d, 2q)=1}}^*L\left(\tfrac{1}{2}+\alpha, f\otimes\chi_{8d}\right)L\left(\tfrac{1}{2}+\beta, f\otimes\chi_{8d}\right)J\left(\frac{8d}{X}\right)\\
&=\frac{X}{2\pi^2}\sum_{\epsilon_1, \epsilon_2\in\{\pm1\}}
(-\eta)^{\delta_1+\delta_2}\gamma_{\alpha}^{\delta_1}\gamma_{\beta}^{\delta_2}X^{-2\delta_1\alpha-2\delta_2\beta}\widetilde{J}(1-2\delta_1\alpha-2\delta_2\beta)
R(\epsilon_1\alpha, \epsilon_2\beta)\\
&\quad+O_q(X(\log\log X)^{4}),
\end{aligned}
\end{equation}
the big-$O$ is depending on $J$ and $q$. Here, $\widetilde{J}(\omega)=\int_{0}^{+\infty}J(x)x^{\omega-1}dx$ is the Mellin transform of $J$, $R(\epsilon_1\alpha, \epsilon_2\beta)=Z_{q'}(\epsilon_1\alpha, \epsilon_2\beta)+U^{-(\epsilon_1\alpha+\epsilon_2\beta)}W_{q', \epsilon_2\beta}(-\epsilon_1\alpha, -\epsilon_2\beta)$ and $Z_{q'}(\epsilon_1\alpha, \epsilon_2\beta)$ is defined by
\begin{equation}\label{eq.Z_q.1}
Z_{q'}(u, v)=\mathop{\sum\sum}_{\substack{q'n_1 n_2=\square \\ \left(n_1 n_2, 2\right)=1}}\frac{\lambda_f\left(n_1\right) \lambda_f\left(n_2\right)}{n_1^{1/2+u} n_2^{1/2+v}}\prod_{p \mid qn_1 n_2}\left(1-\frac{1}{p+1}\right),
\end{equation}
where $q'=1$ or $q'=q$. $W_{q', \epsilon_2\beta}(-\epsilon_1\alpha, -\epsilon_2\beta)$ is defined by
\begin{equation}\label{eq.W}
\begin{aligned}
&\quad W_{q', \epsilon_2\beta}(-\epsilon_1\alpha, -\epsilon_2\beta)\\
&=\frac{g_{\epsilon_2\beta}(-\epsilon_1\alpha-\epsilon_2\beta)}{-\epsilon_1\alpha-\epsilon_2\beta}L(1, \operatorname{sym}^2 f)L(1+2\alpha, \operatorname{sym}^2 f)L(1-2\alpha, \operatorname{sym}^2 f)Z_{q'}^*\left(\tfrac{1}{2}+\alpha, \tfrac{1}{2}-\alpha\right),
\end{aligned}
\end{equation}
$g_{\epsilon_2\beta}$ and $Z_{q'}^*$ are given by formulas \eqref{eq.W,g} and \eqref{eq.z_q'}, respectively. The parameters $\delta_i$ and $q'$ is defined by
\begin{equation*}
q'=
\begin{cases}
1,\;\epsilon_1=\epsilon_2,\\
q,\;\epsilon_1\neq\epsilon_2,
\end{cases}
\quad
\delta_i=
\begin{cases}
0,\;\epsilon_i=1,\\
1,\;\epsilon_i=-1,
\end{cases}
\end{equation*}
for $i=1, 2$ respectively. Finally, the symbol $\gamma_\alpha$, $\gamma_{\beta}$ are
\begin{equation*}
\left(\frac{\sqrt{q}}{2\pi}\right)^{-2\alpha}\frac{\Gamma(1-\alpha)}{\Gamma(1+\alpha)},\quad
\left(\frac{\sqrt{q}}{2\pi}\right)^{-2\beta}\frac{\Gamma(1-\beta)}{\Gamma(1+\beta)}.
\end{equation*}
\end{theorem}
The proof of Theorem \ref{thm3.1} is based on Li's technique. The key is to bound sums of the form
\begin{equation*}
S={\sum_{m\asymp X}}^{*}\left|\sum_{n\asymp N}\frac{\lambda_f(n)}{\sqrt{n}}\left(\frac{m}{n}\right)\right|^2,
\end{equation*}
when $N$ is close to $X$. The influential work of Heath-Brown\cite{HB} implies that $S\ll X^{1+\varepsilon}$, Li showed that $S\ll X$.

We now prove the theorem. Referring to the Iwaniec and Kowalski's book\cite[Theorem 5.3]{I}, the $L$-function has the following approximate functional equation at the critical value.
For any normalized primitive form $f$ with odd level $q$ and weight $2$, we have
\begin{equation*}
\begin{aligned}
L\left(\tfrac{1}{2}+\gamma, f\otimes\chi_{8d}\right)&=\sum_{n=1}^{+\infty}\frac{\lambda_f(n)\chi_{8d}(n)}{n^{\frac{1}{2}+\gamma}}W_{\gamma}\left(\frac{n}{8|d|}\right)\\
&\quad+(-\eta)\chi_{8d}(-q)\left(\frac{8\sqrt{q}|d|}{2\pi}\right)^{-2\gamma}\frac{\Gamma(1-\gamma)}{\Gamma(1+\gamma)}
\sum_{n=1}^{+\infty}\frac{\lambda_f(n)\chi_{8d}(n)}{n^{\frac{1}{2}-\gamma}}W_{-\gamma}\left(\frac{n}{8|d|}\right).
\end{aligned}
\end{equation*}
where
\begin{equation}\label{eq.W,g}
W_{\gamma}(x):=\frac{1}{2\pi i}\int_{(1)}\frac{g_{\gamma}(s)}{s}x^{-s}ds,\quad
g_{\gamma}(s):=\left(\frac{2\pi}{\sqrt{q}}\right)^{-s}\frac{\Gamma(1+\gamma+s)}{\Gamma(1+\gamma)}.
\end{equation}

With hindsight, we set $U=\frac{X}{(\log X)^{1000}}$ and $Y_1=Y_2=(\log X)^{20}$. Further let
\begin{equation*}
\begin{aligned}
\mathcal{A}(\gamma, 8d)&=\sum_{n}\frac{\lambda_f(n)\chi_{8d}(n)}{n^{\frac{1}{2}+\gamma}}W_{\gamma}\left(\frac{n}{U}\right)\\
&\quad+(-\eta)\chi_{8d}(-q)\left(\frac{8\sqrt{q}|d|}{2\pi}\right)^{-2\gamma}\frac{\Gamma(1-\gamma)}{\Gamma(1+\gamma)}
\sum_{n}\frac{\lambda_f(n)\chi_{8d}(n)}{n^{\frac{1}{2}-\gamma}}W_{-\gamma}\left(\frac{n}{U}\right),\\
\mathcal{B}(\gamma, 8d)&=L(1/2+\gamma, f\otimes\chi_{8d})-\mathcal{A}(\gamma, 8d).
\end{aligned}
\end{equation*}
Then we rephrase the goal as
\begin{equation*}
\begin{aligned}
&\quad{\sum_{(d, 2q)=1}}^*L\left(\tfrac{1}{2}+\alpha, f\otimes\chi_{8d}\right)L\left(\tfrac{1}{2}+\beta, f\otimes\chi_{8d}\right)J\left(\frac{8d}{X}\right)\\
&= {\sum_{(d, 2q)=1}}^*\mathcal{A}(\alpha, 8d)L\left(\tfrac{1}{2}+\beta, f\otimes\chi_{8d}\right)J\left(\frac{8d}{X}\right)-{\sum_{(d, 2q)=1}}^*\mathcal{A}(\alpha, 8d)\mathcal{A}(\beta, 8d)J\left(\frac{8d}{X}\right)\\
&\quad+{\sum_{(d, 2q)=1}}^*\mathcal{A}(\beta, 8d)L\left(\tfrac{1}{2}+\alpha, f\otimes\chi_{8d}\right)J\left(\frac{8d}{X}\right)+{\sum_{(d, 2q)=1}}^*\mathcal{B}(\alpha, 8d)\mathcal{B}(\beta, 8d)J\left(\frac{8d}{X}\right)\\
&:=\Rmnum{1}_f-\Rmnum{2}_f+\Rmnum{3}_f+{\sum_{(d, 2q)=1}}^*\mathcal{B}(\alpha, 8d)\mathcal{B}(\beta, 8d)J\left(\frac{8d}{X}\right).
\end{aligned}
\end{equation*}

Our main goal in the rest of the section is to prove the following propositions. The argumentation processes of parts $\Rmnum{1}_f$, $\Rmnum{2}_f$, $\Rmnum{3}_f$ are similar. Therefore, in \S \ref{Proof of Proposition 3.3}, we only provide a detailed proof for part $\Rmnum{1}_f$, the proofs for parts $\Rmnum{2}_f$ and $\Rmnum{3}_f$ are omitted. The desired upper bound for part $\sum_{d}\mathcal{B}(\alpha, 8d)\mathcal{B}(\beta, 8d)$ is obtained by employing a partition of unity together with repeated use of Lemma \ref{lem.SMNt}. For details, see \S \ref{Proof of Proposition 3.2}.

\subsection{The error term}
\label{Proof of Proposition 3.2}

\begin{proposition}\label{Pro3.2}
With notation as above, we have
\begin{equation*}
{\sum_{(d, 2q)=1}}^*\mathcal{B}(\alpha, 8d)\mathcal{B}(\beta, 8d)J\left(\frac{8d}{X}\right)\ll X(\log\log X)^4.
\end{equation*}
\end{proposition}
\begin{proof}
Let $J_{\gamma}(x)=x^{\gamma}J(x)$, by definition of $\mathcal{B}(\gamma, 8d)$, we have
\begin{equation*}
\begin{aligned}
{\sum_{(d, 2q)=1}}^*\mathcal{B}(\alpha, 8d)\mathcal{B}(\beta, 8d)J\left(\frac{8d}{X}\right)
&=\mathcal{M}(\alpha, \beta, 8d)+(-\eta)\gamma_{\beta} X^{-2\beta}\mathcal{M}(\alpha, -\beta, 8d)\\
&+(-\eta)\gamma_{\alpha} X^{-2\alpha}\mathcal{M}(-\alpha, \beta, 8d)
+\gamma_{\alpha}\gamma_{\beta}X^{-2\alpha-2\beta}\mathcal{M}(-\alpha, -\beta, 8d),
\end{aligned}
\end{equation*}
where $\mathcal{M}(\epsilon_1\alpha, \epsilon_2\beta, 8d)$, $\epsilon_1, \epsilon_2\in\{\pm1\}$ are defined as follows
\begin{equation*}
\begin{aligned}
\mathcal{M}(\epsilon_1\alpha, \epsilon_2\beta, 8d)&:={\sum_{(d, 2q)=1}}^*\chi_{8d}(q')J_{2\gamma}\left(\frac{8d}{X}\right)\sum_{n_1\geq1}\frac{\lambda_f(n_1)\chi_{8d}(n_1)}{n_1^{\frac{1}{2}+\epsilon_1\alpha}}
\left[W_{\epsilon_1\alpha}\left(\frac{n_1}{8d}\right)-W_{\epsilon_1\alpha}\left(\frac{n_1}{U}\right)\right]\\
&\quad\times\sum_{n_2\geq1}\frac{\lambda_f(n_2)\chi_{8d}(n_2)}{n_2^{\frac{1}{2}+\epsilon_2\beta}}\left[W_{\epsilon_2\beta}\left(\frac{n_2}{8d}\right)-W_{\epsilon_2\beta}\left(\frac{n_2}{U}\right)\right],
\end{aligned}
\end{equation*}
here, $q'$ and $\gamma$ are given by
\begin{equation}\label{eq.q'-gamma}
q'=
\begin{cases}
q,\quad \epsilon_1\neq\epsilon_2,\\
1,\quad \epsilon_1=\epsilon_2.
\end{cases}
\gamma=
\begin{cases}
-\beta,\quad & (\epsilon_1, \epsilon_2)=(1, -1),\\
-\alpha,\quad & (\epsilon_1, \epsilon_2)=(-1, 1),\\
-\alpha-\beta,\quad & (\epsilon_1, \epsilon_2)=(-1, -1),\\
0,\quad & (\epsilon_1, \epsilon_2)=(1, 1).
\end{cases}
\end{equation}
There exists a smooth real-valued function $G$ compactly supported on $[3/4, 2]$ which satisfies
\begin{equation}\label{eq.G}
\begin{aligned}
G(x)=1 &\text{ for all $x\in[1, 3/2]$},\\
G(x)+G(x/2)=1 &\text{ for all $x\in[1,3]$}.
\end{aligned}
\end{equation}
For more detailed information, readers can refer to Warner's book\cite[Theorem 1.11 and Corollary]{Warner}.
Let $H$ be a positive integer. The function $F(x)$ defined by
\begin{equation}\label{eq.F}
F(x)=G(x)+G(x/2)+\ldots+G\left(x/2^H\right)
\end{equation}
satisfies $F(x)=1$ for all $x\in [1, 3\cdot 2^{H-1}]$ and support on $[3/4, 2^{H+1}]$. In particular, the function
\begin{equation}\label{eq.V}
V(x)=G(2x)+G(x)+G(x/2)
\end{equation}
satisfies $V(x)=1$ for all $x\in[1/2, 3]$. We have the smooth partition of unity
\begin{equation}
\sum_{H=0}^{\infty} G\left(\frac{x}{2^H}\right)=1
\end{equation}
for all $x\in [1,\infty)$.

A dyadic partition of unity using the functions $G$ and $V$ and applying the Mellin inversion for $G(n_1/N_1)$ and $G(n_2/N_2)$, we obtain
\begin{equation*}
\begin{aligned}
&\quad\mathcal{M}(\epsilon_1\alpha, \epsilon_2\beta, 8d)\\
&={\sum_{N_1, N_2}}^d\frac{1}{(2\pi i)^4}
\int_{(c_1)}\int_{(c_2)}\int_{(\varepsilon)}\int_{(\varepsilon)}g_{\epsilon_1\alpha}(w_1)g_{\epsilon_2\beta}(w_2)\widetilde{G}(u-w_1-\epsilon_1\alpha)\widetilde{G}(v-w_2-\epsilon_2\beta)\\
&\quad\times N_{1}^{u-w_1-\epsilon_1\alpha}N_{2}^{v-w_2-\epsilon_2\beta}
{\sum_{(d, 2q)=1}}^*\chi_{8d}(q')\frac{(8|d|)^{w_1}-U^{w_1}}{w_1}\frac{(8|d|)^{w_2}-U^{w_2}}{w_2}\\
&\quad\times\sum_{n_1\geq1}\frac{\lambda_f(n_1)\chi_{8d}(n_1)}{n_1^{\frac{1}{2}+u}}V\left(\frac{n_1}{N_1}\right)
\sum_{n_2\geq1}\frac{\lambda_f(n_2)\chi_{8d}(n_2)}{n_2^{\frac{1}{2}+v}}V\left(\frac{n_2}{N_2}\right)J_{2\gamma}\left(\frac{8d}{X}\right)dudvdw_1dw_2,
\end{aligned}
\end{equation*}
where ${\sum_{N_1, N_2}}^d$ denotes the sum over $N_1=2^{e_1}$, $N_2=2^{e_2}$ for integer $e_1, e_2\geq0$.\\
There are $9$ cases. In each case, we move the contour in both $u$ and $v$ to the line $\Re u, \Re v=0$.\\
(1) When $N_1\leq U$, $N_2\leq U$, we move the contour in $w_1, w_2$ to $\Re w_1=-1+2/\log X$, $\Re w_2=-1+2/\log X$, noting that we pass no poles in the process. In this case, $N_1^{-\Re(\epsilon_1\alpha)}, N_2^{-\Re(\epsilon_2\beta)}\ll 1$.
Then the contribution of $N_1\leq U$, $N_2\leq U$ is
\begin{align*}
&\quad{\sum_{N_1\leq U}}^d\frac{N_1^{1-2/\log X}}{U^{1-2/\log X}}N_1^{-\Re(\epsilon_1\alpha)}{\sum_{N_2\leq U}}^d\frac{N_2^{1-2/\log X}}{U^{1-2/\log X}}N_2^{-\Re(\epsilon_2\beta)}\int_{(-1+2/\log X)}\int_{(-1+2/\log X)}\int_{-\infty}^{+\infty}\int_{-\infty}^{+\infty}\\
&\quad\times{\sum_{(d, 2q)=1}}^*\left|\sum_{n_1\geq1}\frac{\lambda_f(n_1)\chi_{8d}(n_1)}{n_1^{\frac{1}{2}+it_1}}V\left(\frac{n_1}{N_1}\right)\right|
\left|\sum_{n_2\geq1}\frac{\lambda_f(n_2)\chi_{8d}(n_2)}{n_2^{\frac{1}{2}+it_2}}V\left(\frac{n_2}{N_2}\right)\right|J_{2\gamma}\left(\frac{8d}{X}\right)\\
&\quad\times\frac{dt_1}{(1+|w_1+\epsilon_1\alpha|)^{20}}\frac{dt_2}{(1+|w_2+\epsilon_2\beta|)^{20}}\frac{dw_1}{(1+|it_1-w_1-\epsilon_1\alpha|)^{10}}\frac{dw_2}{(1+|it_2-w_2-\epsilon_2\beta|)^{10}}\\
&\ll{\sum_{N_1\leq U}}^d\frac{N_1}{U}{\sum_{N_2\leq U}}^d
\frac{N_2}{U}\int_{(-1+2/\log X)}\int_{(-1+2/\log X)}\int_{-\infty}^{+\infty}\int_{-\infty}^{+\infty}\\
&\quad\times{\sum_{(d, 2q)=1}}^*\left|\sum_{n_1\geq1}\frac{\lambda_f(n_1)\chi_{8d}(n_1)}{n_1^{\frac{1}{2}+it_1}}V\left(\frac{n_1}{N_1}\right)\right|
\left|\sum_{n_2\geq1}\frac{\lambda_f(n_2)\chi_{8d}(n_2)}{n_2^{\frac{1}{2}+it_2}}V\left(\frac{n_2}{N_2}\right)\right|J_{2\gamma}\left(\frac{8d}{X}\right)\\
&\quad\times\frac{dt_1}{(1+|w_1+\epsilon_1\alpha|)^{10}}\frac{dt_2}{(1+|w_2+\epsilon_2\beta|)^{10}}\frac{dw_1}{(1+|t_1|)^{10}}\frac{dw_2}{(1+|t_2|)^{10}}.
\end{align*}
Note that $8d$ is a fundamental discriminant where $(d, 2)=1$ and $d$ is squarefree. Moreover, $V$ is as in \eqref{eq.V}, so we can apply Lemma \ref{lem.SMNt} to conclude that the above expression is bounded by
\begin{equation*}
\begin{aligned}
&{\sum_{N_1\leq U}}^d\frac{N_1}{U}\left[X+N_1\log\left(2+\frac{N_1}{X}\right)\right]^{\frac{1}{2}}
{\sum_{N_2\leq U}}^d\frac{N_2}{U}\left[X+N_2\log\left(2+\frac{N_2}{X}\right)\right]^{\frac{1}{2}}\\
&\ll X.
\end{aligned}
\end{equation*}
(2) When $N_1\leq U$, $U<N_2\leq X$, we move the contour in $w_1, w_2$ to $\Re w_1=-1+2/\log X$, $\Re w_2=0$. In the range, we note that for $\Re w_2=0$, $\frac{(8|d|)^{w_2}-U^{w_2}}{w_2}\ll\log(\frac{8|d|}{U})\ll\log\frac{X}{U}$. Moreover, ${\sum_{U<N_2\leq X}}^d 1\ll\log\frac{X}{U}$. Thus, similar to the above, this range contributes
\begin{equation*}
\begin{aligned}
&\quad\log\frac{X}{U}{\sum_{N_1\leq U}}^d\frac{N_1^{1-2/\log X}}{U^{1-2/\log X}}N_1^{-\Re(\epsilon_1\alpha)}{\sum_{U<N_2\leq X}}^dN_2^{-\Re(\epsilon_2\beta)}
\int_{(0)}\int_{(-1+2/\log X)}\int_{-\infty}^{+\infty}\int_{-\infty}^{+\infty}\\
&\quad\times{\sum_{(d, 2q)=1}}^*\left|\sum_{n_1\geq1}\frac{\lambda_f(n_1)\chi_{8d}(n_1)}{n_1^{\frac{1}{2}+it_1}}V\left(\frac{n_1}{N_1}\right)\right|
\left|\sum_{n_2\geq1}\frac{\lambda_f(n_2)\chi_{8d}(n_2)}{n_2^{\frac{1}{2}+it_2}}V\left(\frac{n_2}{N_2}\right)\right|J_{2\gamma}\left(\frac{8d}{X}\right)\\
&\quad\times\frac{dt_1}{(1+|t_1|)^{10}}\frac{dt_2}{(1+|t_2|)^{10}}\frac{dw_1}{(1+|w_1+\epsilon_1\alpha|)^{10}}\frac{dw_2}{(1+|w_2+\epsilon_2\beta|)^{10}}\\
&\ll X\log\frac{X}{U}{\sum_{N_1\leq U}}^d\frac{N_1}{U}{\sum_{U<N_2\leq X}}^d1\\
&\ll X(\log\log X)^2,
\end{aligned}
\end{equation*}
again using Cauchy--Schwarz and Lemma \ref{lem.SMNt}. \\
(3) When $N_1\leq U$, $N_2>X$, we move the contour in $w_1, w_2$ to $\Re w_1=-1+2/\log X$, $\Re w_2=4$. The contribution of those $N_1\leq U, N_2>X$ is
\begin{align*}
&\quad{\sum_{N_1\leq U}}^d\frac{N_1^{1-2/\log X}}{U^{1-2/\log X}}N_1^{-\Re(\epsilon_1\alpha)}
{\sum_{N_2>X}}^d\left(\frac{X}{N_2}\right)^4N_2^{-\Re(\epsilon_2\beta)}
\int_{(4)}\int_{(-1+2/\log X)}\int_{-\infty}^{+\infty}\int_{-\infty}^{+\infty}\\
&\quad\times{\sum_{(d, 2q)=1}}^*\left|\sum_{n_1\geq1}\frac{\lambda_f(n_1)\chi_{8d}(n_1)}{n_1^{\frac{1}{2}+it_1}}V\left(\frac{n_1}{N_1}\right)\right|
\left|\sum_{n_2\geq1}\frac{\lambda_f(n_2)\chi_{8d}(n_2)}{n_2^{\frac{1}{2}+it_2}}V\left(\frac{n_2}{N_2}\right)\right|J_{2\gamma}\left(\frac{8d}{X}\right)\\
&\quad\times\frac{dt_1}{(1+|t_1|)^{10}}\frac{dt_2}{(1+|t_2|)^{10}}\frac{dw_1}{(1+|w_1+\epsilon_1\alpha|)^{10}}\frac{dw_2}{(1+|w_2+\epsilon_2\beta|)^{10}}\\
&\ll X^{\frac{1}{2}}{\sum_{N_1\leq U}}^d\frac{N_1}{U}
{\sum_{N_2>X}}^dN_2^{-\Re(\epsilon_2\beta)}\left(\frac{X}{N_2}\right)^4\left[N_2\log\left(2+\frac{N_2}{X}\right)\right]^{\frac{1}{2}}\\
&\ll X.
\end{align*}
In the last line, we have used that $\frac{X}{N_2}\log(2+N_2/X)\ll1$ for $N_2>X$ and
$${\sum_{N_2>X}}^dN_2^{-\Re(\epsilon_2\beta)}\left(\frac{X}{N_2}\right)^3\ll 1 .$$\\
(4) When $U<N_1\leq X$, $N_2\leq U$, we move the contour in $w_1, w_2$ to $\Re w_1=0$, $\Re w_2=-1+2/\log X$. Similarly to case (2), the contribution is
$\ll X(\log\log X)^2$.\\
(5) When $U<N_1\leq X$, $U<N_2\leq X$, we move the contour in $w_1, w_2$ to $\Re w_1=0$, $\Re w_2=0$. The contribution is
\begin{align*}
&\quad\left(\log\frac{X}{U}\right)^2{\sum_{U<N_1\leq X}}^dN_1^{-\Re(\epsilon_1\alpha)}{\sum_{U<N_2\leq X}}^dN_2^{-\Re(\epsilon_2\beta)}
\int_{(0)}\int_{(0)}\int_{-\infty}^{+\infty}\int_{-\infty}^{+\infty}\\
&\quad\times{\sum_{(d, 2q)=1}}^*\left|\sum_{n_1\geq1}\frac{\lambda_f(n_1)\chi_{8d}(n_1)}{n_1^{\frac{1}{2}+it_1}}V\left(\frac{n_1}{N_1}\right)\right|
\left|\sum_{n_2\geq1}\frac{\lambda_f(n_2)\chi_{8d}(n_2)}{n_2^{\frac{1}{2}+it_2}}V\left(\frac{n_2}{N_2}\right)\right|J_{2\gamma}\left(\frac{8d}{X}\right)\\
&\quad\times\frac{dt_1}{(1+|t_1|)^{10}}\frac{dt_2}{(1+|t_2|)^{10}}\frac{dw_1}{(1+|w_1+\epsilon_1\alpha|)^{10}}\frac{dw_2}{(1+|w_2+\epsilon_2\beta|)^{10}}\\
&\ll X(\log\log X)^4.
\end{align*}
(6) When $U<N_1\leq X$, $X<N_2$, we move the contour in $w_1, w_2$ to $\Re w_1=0$, $\Re w_2=4$. The range $U<N_1\leq X$, $N_2>X$ contributes
\begin{align*}
&\quad\log\frac{X}{U}{\sum_{U<N_1\leq X}}^dN_1^{-\Re(\epsilon_1\alpha)}{\sum_{N_2>X}}^d\left(\frac{X}{N_2}\right)^4N_2^{-\Re(\epsilon_2\beta)}
\int_{(4)}\int_{(0)}\int_{-\infty}^{+\infty}\int_{-\infty}^{+\infty}\\
&\quad\times{\sum_{(d, 2q)=1}}^*\left|\sum_{n_1\geq1}\frac{\lambda_f(n_1)\chi_{8d}(n_1)}{n_1^{\frac{1}{2}+it_1}}V\left(\frac{n_1}{N_1}\right)\right|
\left|\sum_{n_2\geq1}\frac{\lambda_f(n_2)\chi_{8d}(n_2)}{n_2^{\frac{1}{2}+it_2}}V\left(\frac{n_2}{N_2}\right)\right|J_{2\gamma}\left(\frac{8d}{X}\right)\\
&\quad\times\frac{dt_1}{(1+|t_1|)^{10}}\frac{dt_2}{(1+|t_2|)^{10}}\frac{dw_1}{(1+|w_1+\epsilon_1\alpha|)^{10}}\frac{dw_2}{(1+|w_2+\epsilon_2\beta|)^{10}}\\
&\ll X^{\frac{1}{2}}\left(\log\frac{X}{U}\right)^2{\sum_{N_2>X}}^d\left(\frac{X}{N_2}\right)^4N_2^{-\Re(\epsilon_2\beta)}\left[N_2\log\left(2+\frac{N_2}{X}\right)\right]^{\frac{1}{2}}\\
&\ll X(\log\log X)^2.
\end{align*}
(7) When $N_1>X$, $N_2\leq U$, we move the contour in $w_1, w_2$ to $\Re w_1=4$, $\Re w_2=-1+2/\log X$. The case is similar to case (3), so is $\ll X$.\\
(8) When $N_1>X$, $U<N_2\leq X$, we move the contour in $w_1, w_2$ to $\Re w_1=4$, $\Re w_2=0$. The case is similar to case (6), so is $\ll X(\log\log X)^2$.\\
(9) When $N_1>X$, $N_2>X$, we move the contour in $w_1, w_2$ to $\Re w_1=4$, $\Re w_2=4$.
\begin{align*}
&\quad{\sum_{N_1>X}}^d\left(\frac{X}{N_1}\right)^4N_1^{-\Re(\epsilon_1\alpha)}{\sum_{N_2>X}}^d\left(\frac{X}{N_2}\right)^4N_2^{-\Re(\epsilon_2\beta)}
\int_{(4)}\int_{(4)}\int_{-\infty}^{+\infty}\int_{-\infty}^{+\infty}\\
&\quad\times{\sum_{(d, 2q)=1}}^*\left|\sum_{n_1\geq1}\frac{\lambda_f(n_1)\chi_{8d}(n_1)}{n_1^{\frac{1}{2}+it_1}}V\left(\frac{n_1}{N_1}\right)\right|
\left|\sum_{n_2\geq1}\frac{\lambda_f(n_2)\chi_{8d}(n_2)}{n_2^{\frac{1}{2}+it_2}}V\left(\frac{n_2}{N_2}\right)\right|J_{2\gamma}\left(\frac{8d}{X}\right)\\
&\quad\times\frac{dt_1}{(1+|t_1|)^{10}}\frac{dt_2}{(1+|t_2|)^{10}}\frac{dw_1}{(1+|w_1+\epsilon_1\alpha|)^{10}}\frac{dw_2}{(1+|w_2+\epsilon_2\beta|)^{10}}\\
&\ll{\sum_{N_1>X}}^dN_1^{-\Re(\epsilon_1\alpha)}\left(\frac{X}{N_1}\right)^4\left[N_1\log\left(2+\frac{N_1}{X}\right)\right]^{\frac{1}{2}}{\sum_{N_2>X}}^dN_2^{-\Re(\epsilon_2\beta)}
\left(\frac{X}{N_2}\right)^4\left[N_2\log\left(2+\frac{N_2}{X}\right)\right]^{\frac{1}{2}}\\
&\ll X.
\end{align*}
Therefore, since $X^{-2\alpha}, X^{-2\beta}, X^{-2\alpha-2\beta}\ll 1$, Proposition \ref{Pro3.2} holds.
\end{proof}

\subsection{The estimation of $\Rmnum{1}_f$}
\label{Proof of Proposition 3.3}

\begin{proposition}\label{Pro3.3}
With notation as above, we have
\begin{equation*}
\begin{aligned}
&\quad\Rmnum{1}_f-\Rmnum{2}_f+\Rmnum{3}_f\\ %
&=\frac{X}{2\pi^2}\sum_{\epsilon_1, \epsilon_2\in\{\pm1\}}
(-\eta)^{\delta_1+\delta_2}\gamma_{\alpha}^{\delta_1}\gamma_{\beta}^{\delta_2}X^{-2\delta_1\alpha-2\delta_2\beta}\widetilde{J}(1-2\delta_1\alpha-2\delta_2\beta)
R(\epsilon_1\alpha, \epsilon_2\beta)
+O_q(X),
\end{aligned}
\end{equation*}
where $\gamma_\alpha$, $\gamma_\beta$, $\delta_i$($i=1, 2$), $R(\epsilon_1\alpha, \epsilon_2\beta)$ are defined as in Theorem \ref{thm3.1}. 
\end{proposition}
The estimation method of $\Rmnum{1}_f$, $\Rmnum{2}_f$, $\Rmnum{3}_f$ are similar, and the following only describes the case of $\Rmnum{1}_f$ in detail.\par
Let $J_{\gamma}(x)=x^{\gamma}J(x)$, by definition of $\mathcal{A}(\alpha, 8d)$, we have
\begin{equation*}
\begin{aligned}
&\quad{\sum_{(d, 2q)=1}}^*\mathcal{A}(\alpha, 8d)L\left(\tfrac{1}{2}+\beta, f\otimes\chi_{8d}\right)J\left(\frac{8d}{X}\right)\\
&=\mathcal{N}(\beta, \alpha, 8d)+(-\eta)\gamma_{\beta} X^{-2\beta}\mathcal{N}(-\beta, \alpha, 8d)
+(-\eta)\gamma_{\alpha} X^{-2\alpha}\mathcal{N}(\beta, -\alpha, 8d)\\
&\quad+\gamma_{\alpha}\gamma_{\beta}X^{-2\alpha-2\beta}\mathcal{N}(-\beta, -\alpha, 8d),
\end{aligned}
\end{equation*}
and $\mathcal{N}(\epsilon_2\beta, \epsilon_1\alpha, 8d)$, $\epsilon_1, \epsilon_2\in\{\pm1\}$ is defined as follows
\begin{equation*}
\mathcal{N}(\epsilon_2\beta, \epsilon_1\alpha, 8d)={\sum_{(d, 2q)=1}}^*J_{2\gamma}\left(\frac{8d}{X}\right)\mathop{\sum\sum}_{n_1, n_2\geq1}\frac{\lambda_f(n_1)\lambda_f(n_2)}{n_1^{\frac{1}{2}+\epsilon_2\beta}n_2^{\frac{1}{2}+\epsilon_1\alpha}}\chi_{8d}(q'n_1n_2)W_{\epsilon_2\beta}\left(\frac{n_1}{8d}\right)W_{\epsilon_1\alpha}\left(\frac{n_2}{U}\right),
\end{equation*}
where $q'$ and $\gamma$ are shown as formula \eqref{eq.q'-gamma}.
When $\epsilon_1\neq\epsilon_2$, we have that
\begin{equation*}
\begin{aligned}
&\quad\mathcal{N}(\pm\beta, \mp\alpha, 8d)\\
&=\sum_{(d, 2q)=1}\sum_{a^2\mid d}\mu(a)J_{2\gamma}\left(\frac{8d}{X}\right)\mathop{\sum\sum}_{n_1, n_2\geq1}
\frac{\lambda_f(n_1)\lambda_f(n_2)}{n_1^{\frac{1}{2}\pm\beta}n_2^{\frac{1}{2}\mp\alpha}}\chi_{8d}(qn_1n_2)W_{\pm\beta}\left(\frac{n_1}{8d}\right)W_{\mp\alpha}\left(\frac{n_2}{U}\right).
\end{aligned}
\end{equation*}
Exchanging the order of summation, we see that
\begin{equation*}
\begin{aligned}
&\quad\mathcal{N}(\pm\beta, \mp\alpha, 8d)\\
&=\sum_{(a, 2q)=1}\mu(a)\mathop{\sum\sum}_{(n_1n_2, a)=1}
\frac{\lambda_f(n_1)\lambda_f(n_2)}{n_1^{\frac{1}{2}\pm\beta}n_2^{\frac{1}{2}\mp\alpha}}W_{\mp\alpha}\left(\frac{n_2}{U}\right)\sum_{(d, 2)=1}\chi_{8d}(qn_1n_2)J_{2\gamma}\left(\frac{8a^2d}{X}\right)W_{\pm\beta}\left(\frac{n_1}{8a^2d}\right)\\
&=\left(\sum_{\substack{(a, 2q)=1\\a\leq Y_1}}+\sum_{\substack{(a, 2q)=1\\a>Y_1}}\right)\mu(a)
\mathop{\sum\sum}_{(n_1n_2, a)=1}\frac{\lambda_f(n_1)\lambda_f(n_2)}{n_1^{\frac{1}{2}\pm\beta}n_2^{\frac{1}{2}\mp\alpha}}W_{\mp\alpha}\left(\frac{n_2}{U}\right)\\
&\quad\times\sum_{(d, 2)=1}\chi_{8d}(qn_1n_2)J_{2\gamma}\left(\frac{8a^2d}{X}\right)W_{\pm\beta}\left(\frac{n_1}{8a^2d}\right).
\end{aligned}
\end{equation*}
Similarly, when $\epsilon_1=\epsilon_2$, we have that
\begin{align*}
&\quad\mathcal{N}(\pm\beta, \pm\alpha, 8d)\\
&=\sum_{(a, 2q)=1}\mu(a)\sum_{b\mid q}\mu(b)\\
&\quad\times\mathop{\sum\sum}_{(n_1n_2, a)=1}\frac{\lambda_f(n_1)\lambda_f(n_2)}{n_1^{\frac{1}{2}\pm\beta}n_2^{\frac{1}{2}\pm\alpha}}W_{\pm\alpha}\left(\frac{n_2}{U}\right)\sum_{(d, 2)=1}\chi_{8bd}(n_1n_2)J_{2\gamma}\left(\frac{8a^2bd}{X}\right)W_{\pm\beta}\left(\frac{n_1}{8a^2bd}\right)\\
&=\left(\sum_{\substack{(a, 2q)=1\\a\leq Y_1}}+\sum_{\substack{(a, 2q)=1\\a>Y_1}}\right)\mu(a)\left(\sum_{\substack{b\mid q\\ b\leq Y_2}}+\sum_{\substack{b\mid q\\ b>Y_2}}\right)\mu(b)\\
&\quad\times\mathop{\sum\sum}_{(n_1n_2, a)=1}\frac{\lambda_f(n_1)\lambda_f(n_2)}{n_1^{\frac{1}{2}\pm\beta}n_2^{\frac{1}{2}\pm\alpha}}W_{\pm\alpha}\left(\frac{n_2}{U}\right)
\sum_{(d, 2)=1}\chi_{8bd}(n_1n_2)J_{2\gamma}\left(\frac{8a^2bd}{X}\right)W_{\pm\beta}\left(\frac{n_1}{8a^2bd}\right).
\end{align*}

We shall illustrate $\mathcal{N}(\beta, -\alpha, 8d)$ in detail. Based on the above operations, we proceed by considering two cases: $a > Y_1$ and $a \leq Y_1$. The contribution for $a > Y_1$ is negligible, see \S \ref{subsubsection: The contribution of a>Y_1} for details. By the Poisson summation formula, the case $a\leq Y_1$ can be decomposed into the diagonal and off-diagonal contributions. The diagonal contribution arising from the terms when $qn_1n_2$ is perfect square dominates, generically $G_k(qn_1n_2)=0$ when $qn_1n_2$ is not squarefree, so that the same diagonal contribution simply does not exist. The detailed proof is provided in \S \ref{subsubsection: The contribution of diagonal, square term}. The analysis of the off-diagonal terms is more involved, and we defer the details to \S \ref{subsubsection: The contribution of off-diagonal term}.

The case of $\mathcal{N}(\beta, \alpha, 8d)$, $\mathcal{N}(-\beta, \alpha, 8d)$ and $\mathcal{N}(-\beta, -\alpha, 8d)$ can be handled similarly. We will not elaborate further.
\subsubsection{The contribution of $a>Y_1$}
\label{subsubsection: The contribution of a>Y_1}

\begin{lemma}\label{lem.T}
With notation as above, for any $\varepsilon>0$, we have the following bounds: if $\epsilon_1=\epsilon_2$,
\begin{equation*}
\mathcal{N}(\epsilon_2\beta, \epsilon_1\alpha, 8d)\ll_q
\begin{cases}
\frac{X}{Y_2},\;a\leq Y_1, b>Y_2,\\
\frac{X}{Y_1^{1-\varepsilon}},\;a> Y_1, b\leq Y_2,\\
\frac{X}{Y_1^{1-\varepsilon}Y_2},\;a> Y_1, b>Y_2;
\end{cases}
\end{equation*}
if $\epsilon_1\neq\epsilon_2$, $\mathcal{N}(\epsilon_2\beta, \epsilon_1\alpha, 8d)(a>Y_1)\ll_q\frac{X}{Y_1^{1-\varepsilon}}$.
\end{lemma}
\begin{proof} We will focus on elaborating the details of $\mathcal{N}(\beta, -\alpha, 8d)(a>Y_1)$. The others can be handled similarly. We begin by recalling the definition of $\mathcal{N}(\beta, -\alpha, 8d)(a>Y_1)$.
\begin{equation*}
\begin{aligned}
&\quad\mathcal{N}(\beta, -\alpha, 8d)(a>Y_1)\\
&=\sum_{\substack{(a, 2q)=1\\a>Y_1}}\mu(a)\sum_{(d, 2)=1}J_{-2\alpha}\left(\frac{8a^2d}{X}\right)\mathop{\sum\sum}_{(n_1n_2, a)=1}\frac{\lambda_f(n_1)\lambda_f(n_2)}{n_1^{\frac{1}{2}+\beta}n_2^{\frac{1}{2}-\alpha}}\chi_{8d}(qn_1n_2)W_{\beta}\left(\frac{n_1}{8a^2d}\right)W_{-\alpha}\left(\frac{n_2}{U}\right).
\end{aligned}
\end{equation*}
A dyadic partition of unity using the functions $G$ and $V$ (see formulas \eqref{eq.G} and \eqref{eq.V} respectively) and applying the Mellin inversion for $G(n_1/N_1)$ and $G(n_2/N_2)$, the second-layer summation is expressed as,
\begin{align*}
&\quad\sum_{(d, 2)=1}\chi_{8d}(q)J_{-2\alpha}\left(\frac{8a^2d}{X}\right)\mathop{\sum\sum}_{(n_1n_2, a)=1}
\frac{\lambda_f(n_1)\lambda_f(n_2)}{n_1^{\frac{1}{2}+\beta}n_2^{\frac{1}{2}-\alpha}}\chi_{8d}(n_1n_2)W_{\beta}\left(\frac{n_1}{8a^2d}\right)W_{-\alpha}\left(\frac{n_2}{U}\right)\\
&=\frac{1}{(2\pi i)^4}{\sum_{N_1, N_2}}^{d}\int_{(\varepsilon)}\int_{(\varepsilon)}\int_{(\varepsilon)}\int_{(\varepsilon)}\frac{g_{-\alpha}(u)g_{\beta}(v)}{uv}U^u
\sum_{(d, 2)=1}\chi_{8d}(q)(8a^2d)^v\\
&\quad\times\sum_{(n_1, a)=1}\frac{\lambda_f(n_1)}{n^{1/2+w_1}}\chi_{8d}(n_1)V\left(\frac{n_1}{N_1}\right)
\sum_{(n_2, a)=1}\frac{\lambda_f(n_2)}{n^{1/2+w_2}}\chi_{8d}(n_2)V\left(\frac{n_2}{N_2}\right)J_{-2\alpha}\left(\frac{8a^2d}{X}\right)\\
&\quad\times N_{1}^{w_1-v-\beta}N_{2}^{w_2-u+\alpha}\widetilde{G}(w_1-v-\beta)\widetilde{G}(w_2-u+\alpha)dudvdw_1dw_2,
\end{align*}
where ${\sum_{N_1, N_2}}^{d}$ denotes sum over $N_1=2^{e_1}$, $N_2=2^{e_2}$ for integers $e_1, e_2\geq 0$. We split this sum into four distinct parts.

(1) When $N_1\leq X$, $N_2\leq U$, we now shift the contour of integration in $u$, $v$ to $\Re u=-1+2/\log X$, $\Re v=-1+2/\log X$. We move the contour in $w_1$, $w_2$ to $\Re w_1=0$, $\Re w_2=0$. This part is
\begin{align*}
&\ll{\sum_{N_1\leq X}}^d\frac{N_1^{1-2/\log X}}{X^{1-2/\log X}}N_1^{-\Re\beta}{\sum_{N_2\leq U}}^d\frac{N_2^{1-2/\log X}}{U^{1-2/\log X}}N_2^{\Re\alpha}
\int_{-\infty}^{+\infty}\int_{-\infty}^{+\infty}\int_{(-1+2/\log X)}\int_{(-1+2/\log X)}\\
&\quad\times\frac{1}{(1+|u-\alpha|)^{20}}\frac{1}{(1+|v+\beta|)^{20}}\frac{1}{(1+|it_1-(v+\beta)|)^{10}}\frac{1}{(1+|it_2-(u-\alpha)|)^{10}}\\
&\quad\times\sum_{(d, 2)=1}\left|\sum_{(n_1, a)=1}\frac{\lambda_f(n_1)}{n^{1/2+it_1}}\chi_{8d}(n_1)V\left(\frac{n_1}{N_1}\right)\right|
\left|\sum_{(n_2, a)=1}\frac{\lambda_f(n_2)}{n^{1/2+it_2}}\chi_{8d}(n_2)V\left(\frac{n_2}{N_2}\right)\right|\\
&\quad\times J_{-2\alpha}\left(\frac{8a^2d}{X}\right)dudvdt_1dt_2.
\end{align*}
In this case, $N_1^{-\Re\beta}, N_2^{\Re\alpha}\ll1$, ${\sum_{N_1\leq X}}^d\frac{N_1}{X}, {\sum_{N_2\leq U}}^d\frac{N_2}{U}\ll 1$. According to Lemma \ref{lem.6.3} and Cauchy--Schwarz inequality, we have
\begin{equation*}
\ll\frac{\tau(a)^5}{a^2}X.
\end{equation*}
(2) When $N_1\leq X$, $N_2> U$, we shift the contour of integration in $u$, $v$ to $\Re u=1$, $\Re v=-1+2/\log X$. We move the contour in $w_1$, $w_2$ to $\Re w_1=0$, $\Re w_2=0$. Similar to the above, the contribution of $N_1\leq X$, $N_2> U$ is
\begin{equation*}
\begin{aligned}
&\ll{\sum_{N_1\leq X}}^d\frac{N_1}{X}N_1^{-\Re\beta}{\sum_{N_2> U}}^d\frac{U}{N_2}N_2^{\Re\alpha}
\int_{-\infty}^{+\infty}\int_{-\infty}^{+\infty}\int_{(-1+2/\log X)}\int_{(1)}\\
&\quad\times\frac{1}{(1+|u-\alpha|)^{20}}\frac{1}{(1+|v+\beta|)^{20}}\frac{1}{(1+|it_1-(v+\beta)|)^{10}}\frac{1}{(1+|it_2-(u-\alpha)|)^{10}}\\
&\quad\times\sum_{(d, 2)=1}\left|\sum_{(n_1, a)=1}\frac{\lambda_f(n_1)}{n^{1/2+it_1}}\chi_{8d}(n_1)V\left(\frac{n_1}{N_1}\right)\right|
\left|\sum_{(n_2, a)=1}\frac{\lambda_f(n_2)}{n^{1/2+it_2}}\chi_{8d}(n_2)V\left(\frac{n_2}{N_2}\right)\right|\\
&\quad\times J_{-2\alpha}\left(\frac{8a^2d}{X}\right)dudvdt_1dt_2\\
&\ll\frac{\tau(a)^5}{a^2}X
\end{aligned}
\end{equation*}
again using Lemma \ref{lem.6.3} and ${\sum_{N_2> U}}^d\frac{U}{N_2}N_2^{\Re\alpha}\ll 1$. \\
(3) When $N_1> X$, $N_2\leq U$, we shift the contour of integration in $u$, $v$ to $\Re u=-1+2/\log X$, $\Re v=1$. We move the contour in $w_1$, $w_2$ to $\Re w_1=0$, $\Re w_2=0$. The range $N_1> X$, $N_2\leq U$ contributes
\begin{equation*}
\begin{aligned}
&\ll{\sum_{N_1>X}}^d\frac{X}{N_1}N_1^{-\Re\beta}{\sum_{N_2\leq U}}^d\frac{N_2^{1-2/\log X}}{U^{1-2/\log X}}N_2^{\Re\alpha}
\int_{-\infty}^{+\infty}\int_{-\infty}^{+\infty}\int_{(1)}\int_{(-1+2/\log X)}\\
&\quad\times\frac{1}{(1+|u-\alpha|)^{20}}\frac{1}{(1+|v+\beta|)^{20}}\frac{1}{(1+|it_1-(v+\beta)|)^{10}}\frac{1}{(1+|it_2-(u-\alpha)|)^{10}}\\
&\quad\times\sum_{(d, 2)=1}\left|\sum_{(n_1, a)=1}\frac{\lambda_f(n_1)}{n^{1/2+it_1}}\chi_{8d}(n_1)V\left(\frac{n_1}{N_1}\right)\right|
\left|\sum_{(n_2, a)=1}\frac{\lambda_f(n_2)}{n^{1/2+it_2}}\chi_{8d}(n_2)V\left(\frac{n_2}{N_2}\right)\right|\\
&\quad\times J_{-2\alpha}\left(\frac{8a^2d}{X}\right)dudvdt_1dt_2\\
&\ll\frac{\tau(a)^5}{a^2}X.
\end{aligned}
\end{equation*}
(4) Finally, when $N_1>X$, $N_2>U$, we now shift the contour of integration in $u$, $v$ to $\Re u=1$, $\Re v=1$. We move the contour in $w_1$, $w_2$ to $\Re w_1=0$, $\Re w_2=0$.
According to Lemma \ref{lem.6.3} and Cauchy--Schwarz, this part is $\ll\frac{\tau(a)^5}{a^2}X$.
Hence,
\begin{equation*}
\mathcal{N}(\beta, -\alpha, 8d)(a>Y_1)\ll\sum_{\substack{(a, 2q)=1\\a>Y_1}}\frac{\tau(a)^5}{a^2}X\ll\frac{X}{Y_1^{1-\varepsilon}}.
\end{equation*}
This completes the proof of the lemma.
\end{proof}
\subsubsection{The contribution of $a\leq Y_1$}
\label{subsubsection: The contribution of a Y_1}

Next, we will handle the part where $a\leq Y_1$. Let $\Phi(x)=J_{-2\alpha}(x)W_{\beta}(n_1/xX)$, it follows from Poisson summation formula that
\begin{equation*}
\begin{aligned}
&\quad\sum_{\substack{(a, 2q)=1\\a\leq Y_1}}\mu(a)\mathop{\sum\sum}_{(n_1n_2, a)=1}
\frac{\lambda_f(n_1)\lambda_f(n_2)}{n_1^{\frac{1}{2}+\beta}n_2^{\frac{1}{2}-\alpha}}W_{-\alpha}\left(\frac{n_2}{U}\right)
\sum_{(d, 2)=1}\chi_{8d}(qn_1n_2)J_{-2\alpha}\left(\frac{8a^2d}{X}\right)W_{\beta}\left(\frac{n_1}{8a^2d}\right)\\
&=\frac{X}{16}\sum_{\substack{(a, 2q)=1\\a\leq Y_1}}\frac{\mu(a)}{a^2}\mathop{\sum\sum}_{(n_1n_2, 2a)=1}\frac{\lambda_f(n_1)\lambda_f(n_2)}{n_1^{\frac{1}{2}+\beta}n_2^{\frac{1}{2}-\alpha}}
\sum_{k\in\mathbb{Z}}(-1)^k\frac{G_k(qn_1n_2)}{qn_1n_2}\check{\Phi}\left(\frac{kX}{16a^2qn_1n_2}\right)\\
&=\frac{X}{16}\sum_{\substack{(a, 2q)=1\\a\leq Y_1}}\frac{\mu(a)}{a^2}\mathop{\sum\sum}_{\substack{(n_1n_2, 2a)=1\\ qn_1n_2=\square}}
\frac{\lambda_f(n_1)\lambda_f(n_2)}{n_1^{\frac{1}{2}+\beta}n_2^{\frac{1}{2}-\alpha}}\prod_{p\mid qn_1n_2}\left(1-\frac{1}{p}\right)W_{-\alpha}\left(\frac{n_2}{U}\right)\check{\Phi}(0)\\
&\quad+\frac{X}{16}\sum_{\substack{(a, 2q)=1\\a\leq Y_1}}\frac{\mu(a)}{a^2}\mathop{\sum\sum}_{(n_1n_2, 2a)=1}
\frac{\lambda_f(n_1)\lambda_f(n_2)}{n_1^{\frac{1}{2}+\beta}n_2^{\frac{1}{2}-\alpha}}W_{-\alpha}\left(\frac{n_2}{U}\right)
\sum_{k\neq0}(-1)^k\frac{G_k(qn_1n_2)}{qn_1n_2}\check{\Phi}\left(\frac{kX}{16a^2qn_1n_2}\right).
\end{aligned}
\end{equation*}
Here, $\check{\Phi}$ is defined as follows
\begin{equation*}
\check{\Phi}\left(\frac{kX}{16a^2qn_1n_2}\right)=\int_{-\infty}^{+\infty}(\cos+\sin)\left(\frac{2\pi x\cdot kX}{16a^2qn_1n_2}\right)J_{-2\alpha}(x)W_{\beta}\left(\frac{n_1}{xX}\right)dx.
\end{equation*}

\paragraph{The contribution of diagonal, square term}
\label{subsubsection: The contribution of diagonal, square term}

\begin{lemma}
For $Y_1\geq(\log X)^{11}$, the case $k=0$ yields the following asymptotic estimate,
\begin{equation}\label{eq.6.4}
\begin{aligned}
&\quad\frac{X}{16}\sum_{\substack{(a, 2q)=1\\a\leq Y_1}}\frac{\mu(a)}{a^2}\mathop{\sum\sum}_{\substack{(n_1n_2, 2a)=1\\ qn_1n_2=\square}}
\frac{\lambda_f(n_1)\lambda_f(n_2)}{n_1^{\frac{1}{2}+\beta}n_2^{\frac{1}{2}-\alpha}}\prod_{p\mid qn_1n_2}\left(1-\frac{1}{p}\right)W_{-\alpha}\left(\frac{n_2}{U}\right)\check{\Phi}(0)\\
&=\frac{X}{2\pi^2}\widetilde{J}(1-2\alpha)\zeta(1+\beta-\alpha)L(1+2\beta, \operatorname{sym}^2 f) L(1-2\alpha, \operatorname{sym}^2 f)L(1+\beta-\alpha, \operatorname{sym}^2 f)\\
&\quad\times Z_{q}^*\left(\tfrac{1}{2}+\beta, \tfrac{1}{2}-\alpha\right)
+\frac{XU^{\alpha-\beta}}{2\pi^2}\widetilde{J}(1-2\alpha)\frac{g_{-\alpha}(\alpha-\beta)}{\alpha-\beta}L(1, \operatorname{sym}^2 f)L(1+2\beta, \operatorname{sym}^2 f)\\
&\quad\times L(1-2\beta, \operatorname{sym}^2 f)Z_{q}^*\left(\tfrac{1}{2}+\beta, \tfrac{1}{2}-\beta\right)
+O_q(X/\log X),
\end{aligned}
\end{equation}
where the big-$O$ depends on $q$.
\end{lemma}
\begin{proof}
Note that
\begin{equation}
\begin{aligned}
\sum_{\substack{a \leq Y_1 \\ \left(a, 2 n_1 n_2q\right)=1}} \frac{\mu(a)}{a^2}
& =\frac{1}{\zeta(2)} \prod_{p \mid 2 n_1 n_2q}\left(1-\frac{1}{p^2}\right)^{-1}+O(1/Y_1) \\
& =\frac{8}{\pi^2} \prod_{p \mid n_1 n_2q}\left(1-\frac{1}{p^2}\right)^{-1}+O(1/Y_1).
\end{aligned}
\end{equation}
Switching the order of summation gives that the left side of (\ref{eq.6.4}) is
\begin{equation*}
\begin{aligned}
&\frac{X}{2\pi^2}\mathop{\sum\sum}_{\substack{(n_1n_2, 2)=1\\ qn_1n_2=\square}}\frac{\lambda_f(n_1)\lambda_f(n_2)}{n_1^{\frac{1}{2}+\beta}n_2^{\frac{1}{2}-\alpha}}
\prod_{p\mid qn_1n_2}\frac{p}{p+1}W_{-\alpha}\left(\frac{n_2}{U}\right)\int_{0}^{+\infty}J_{-2\alpha}(x)W_{\beta}\left(\frac{n_1}{xX}\right)dx\\
&+O\left(\frac{X}{Y_1}\mathop{\sum\sum}_{ qn_1n_2=\square}
\frac{\tau(n_1)\tau(n_2)}{n_1^{\frac{1}{2}+\Re\beta}n_2^{\frac{1}{2}-\Re\alpha}}W_{-\alpha}\left(\frac{n_2}{U}\right)
\left|\int_{0}^{+\infty}J_{-2\alpha}(x)W_{\beta}\left(\frac{n_1}{xX}\right)dx\right|\right).
\end{aligned}
\end{equation*}
The error term inside the big-$O$ is $\ll_q \frac{X}{Y_1}(\log X)^{10}$, it is absorbed by $O_q(X/\log X)$. Let
\begin{equation*}
Z_{q'}(u, v)=\mathop{\sum\sum}_{\substack{q'n_1 n_2=\square \\ \left(n_1 n_2, 2\right)=1}}\frac{\lambda_f\left(n_1\right) \lambda_f\left(n_2\right)}{n_1^{1/2+u} n_2^{1/2+v}}\prod_{p \mid qn_1 n_2}\left(1-\frac{1}{p+1}\right),
\end{equation*}
where $q'=1$ or $q'=q$. Hence, we have
\begin{equation}\label{eq.z_q'}
\begin{aligned}
Z_{q'}(u, v)= & \zeta(1+u+v)L\left(1+2u, \operatorname{sym}^2 f\right) \\
& \times L\left(1+2v, \operatorname{sym}^2 f\right)L\left(1+u+v, \operatorname{sym}^2 f\right)Z^*_{q'}\left(\tfrac{1}{2}+u, \tfrac{1}{2}+v\right)
\end{aligned}
\end{equation}
where $Z^*_{q'}(1/2+u, 1/2+v)$ converges absolutely in the region $\Re u,v \geq -1/4+\varepsilon$. The proofs refer to
Lemma $5.5$ in \cite{Li}.
Then the main term is
\begin{equation}\label{eq3.0}
\frac{X}{2\pi^2}\frac{1}{(2\pi i)^2}\int_{(c_1)}\int_{(c_2)}\frac{g_{-\alpha}(u)g_{\beta}(v)}{uv}U^uX^vZ_q\left(\beta+v, -\alpha+u\right)\widetilde{J}(v+1-2\alpha)dudv.
\end{equation}
 Shifting the contour of integration to $\Re u=1/10$, $\Re v=-1/5$ in (\ref{eq3.0}), we encounter two poles along the way: $v=0, v=-u+\alpha-\beta$. The contribution of the remaining integral at $\Re u=1/10$, $\Re v=-1/5$ contributes $\ll X^{\frac{4}{5}}U^{\frac{1}{10}}$. The contribution of $v=0$ is
\begin{equation}\label{eq.3.3}
\begin{aligned}
&\frac{X}{2\pi^2}\frac{1}{2\pi i}\int_{(1/10)}\frac{g_{-\alpha}(u)}{u}U^u\widetilde{J}(1-2\alpha)\zeta(1+u+\beta-\alpha)L(1+2\beta, \operatorname{sym}^2 f)\\
&\times L(1+2u-2\alpha, \operatorname{sym}^2 f)L(1+u+\beta-\alpha, \operatorname{sym}^2 f)Z_q^*\left(\tfrac{1}{2}+\beta, \tfrac{1}{2}-\alpha+u\right)du.
\end{aligned}
\end{equation}
The contribution of $v=-u+\alpha-\beta$ is
\begin{equation*}
\begin{aligned}
&\frac{X}{2\pi^2}\frac{1}{2\pi i}\int_{(1/10)}\frac{g_{-\alpha}(u)g_{\beta}(-u+\alpha-\beta)}{u(-u+\alpha-\beta)}U^uX^{-u+\alpha-\beta}
\widetilde{J}(1-u-\alpha-\beta)L(1, \operatorname{sym}^2 f)\\
&\times L(1+2u-2\alpha, \operatorname{sym}^2 f)L(1-2u+2\alpha, \operatorname{sym}^2 f)Z_q^*\left(\tfrac{1}{2}-u+\alpha, \tfrac{1}{2}+u-\alpha\right)du\\
&\ll X/\log X.
\end{aligned}
\end{equation*}
For \eqref{eq.3.3}, We similarly shift to $\Re u=-1/5+\varepsilon$, the contour giving a contribution of $\ll XU^{-\frac{1}{5}+\varepsilon}$, while the residue at $u=0$ gives
\begin{equation*}
\begin{aligned}
&\frac{X}{2\pi^2}\widetilde{J}(1-2\alpha)\zeta(1+\beta-\alpha)L(1+2\beta, \operatorname{sym}^2 f)\\
&\quad\times L(1-2\alpha, \operatorname{sym}^2 f)L(1+\beta-\alpha, \operatorname{sym}^2 f)Z_{q}^*\left(\tfrac{1}{2}+\beta, \tfrac{1}{2}-\alpha\right),
\end{aligned}
\end{equation*}
the residue at $u=\alpha-\beta$ is
\begin{equation*}
\begin{aligned}
\frac{XU^{\alpha-\beta}}{2\pi^2}\widetilde{J}(1-2\alpha)\frac{g_{-\alpha}(\alpha-\beta)}{\alpha-\beta}L(1, \operatorname{sym}^2 f)L(1+2\beta, \operatorname{sym}^2 f)L(1-2\beta, \operatorname{sym}^2 f)Z_{q}^*\left(\tfrac{1}{2}+\beta, \tfrac{1}{2}-\beta\right).
\end{aligned}
\end{equation*}
This completes the proof of the lemma.
\end{proof}

\paragraph{The contribution of off-diagonal term}
\label{subsubsection: The contribution of off-diagonal term}

Suppose $2^l\|k$, $k=2^lk'$, $(k', 2)=1$. Since $G_{4k}(n)=G_{k}(n)$ for odd $n$, $G_{2^lk'}(n)=G_{2^\delta k'}(n)$ where $\delta=0$ if $2\mid l$ and $\delta=1$ if $2\nmid l$. Then
\begin{align*}
&\quad\mathop{\sum\sum}_{(n_1n_2, 2a)=1}
\frac{\lambda_f(n_1)\lambda_f(n_2)}{n_1^{1/2+\beta}n_2^{1/2-\alpha}}W_{-\alpha}\left(\frac{n_2}{U}\right)\sum_{k\neq0}(-1)^k\frac{G_k(qn_1n_2)}{qn_1n_2}\check{\Phi}\left(\frac{kX}{16a^2qn_1n_2}\right)\\
&=\sum_{\substack{2^lk\neq0\\(k, 2)=1}}(-1)^{2^lk}\mathop{\sum\sum}_{(n_1n_2, 2a)=1}\frac{\lambda_f(n_1)\lambda_f(n_2)}{n_1^{1/2+\beta}n_2^{1/2-\alpha}}\frac{G_{2^lk}(qn_1n_2)}{qn_1n_2}
\check{\Phi}\left(\frac{2^lkX}{16a^2qn_1n_2}\right)W_{-\alpha}\left(\frac{n_2}{U}\right)\\
&=\sum_{l\geq0}\sum_{\substack{k\neq0\\(k, 2)=1}}(-1)^{2^lk}\mathop{\sum\sum}_{(n_1n_2, 2a)=1}\frac{\lambda_f(n_1)\lambda_f(n_2)}{n_1^{1/2+\beta}n_2^{1/2-\alpha}}\frac{G_{2^lk}(qn_1n_2)}{qn_1n_2}
\check{\Phi}\left(\frac{2^lkX}{16a^2qn_1n_2}\right)W_{-\alpha}\left(\frac{n_2}{U}\right)\\
&=-\sum_{\substack{k\neq0\\(k, 2)=1}}\mathop{\sum\sum}_{(n_1n_2, 2a)=1}\frac{\lambda_f(n_1)\lambda_f(n_2)}{n_1^{1/2+\beta}n_2^{1/2-\alpha}}\frac{G_{k}(qn_1n_2)}{qn_1n_2}
\check{\Phi}\left(\frac{kX}{16a^2qn_1n_2}\right)W_{-\alpha}\left(\frac{n_2}{U}\right)\\
&\quad+\sum_{l\geq1}\sum_{\substack{k\neq0\\(k, 2)=1}}\mathop{\sum\sum}_{(n_1n_2, 2a)=1}\frac{\lambda_f(n_1)\lambda_f(n_2)}{n_1^{1/2+\beta}n_2^{1/2-\alpha}}\frac{G_{2^lk}(qn_1n_2)}{qn_1n_2}
\check{\Phi}\left(\frac{2^lkX}{16a^2qn_1n_2}\right)W_{-\alpha}\left(\frac{n_2}{U}\right).
\end{align*}
We assume that $|\gamma_1|, |\gamma_2|, |\gamma|\leq(\log X)^{-1}$. Let 
\begin{equation*}
T_{l}=\sum_{\substack{k\neq0\\(k, 2)=1}}\mathop{\sum\sum}_{(n_1n_2, 2a)=1}\frac{\lambda_f(n_1)\lambda_f(n_2)}{n_1^{1/2+\gamma_1}n_2^{1/2+\gamma_2}}\frac{G_{2^lk}(qn_1n_2)}{qn_1n_2}
\check{\Phi}\left(\frac{2^lkX}{16a^2qn_1n_2}\right)W_{\gamma_2}\left(\frac{n_2}{U}\right),
\end{equation*}
the Fourier-type transform is given by
\begin{equation*}
\check{\Phi}\left(\frac{2^lkX}{16a^2qn_1n_2}\right)=\int_{-\infty}^{+\infty}(\cos+\sin)\left(\frac{2\pi x\cdot 2^lkX}{16a^2qn_1n_2}\right)
J_{2\gamma}(x)W_{\gamma_1}\left(\frac{n_1}{xX}\right)dx.
\end{equation*}
A dyadic partition of unity using the functions $G$, $V$ as shown in \eqref{eq.G}, \eqref{eq.V} and applying the Mellin inversion for $G(n_1/N_1)$ and $G(n_2/N_2)$, we obtain
\begin{equation*}
\begin{aligned}
T_{l}&={\sum_{N_1, N_2}}^d\frac{1}{(2\pi i)^4}\int_{(c_1)}\int_{(c_2)}\int_{-\infty}^{+\infty}\int_{-\infty}^{+\infty}\widetilde{G}(it_1-u-\gamma_1)\widetilde{G}(it_2-v-\gamma_2)\\
&\quad\times N_1^{it_1-u-\gamma_1}N_2^{it_2-v-\gamma_2}\frac{g_{\gamma_1}(u)g_{\gamma_2}(v)}{uv}X^uU^v\sum_{\substack{k\neq0\\(k, 2)=1}}\mathop{\sum\sum}_{(n_1n_2, 2a)=1}
\frac{\lambda_f(n_1)\lambda_f(n_2)}{n_1^{1/2+it_1}n_2^{1/2+it_2}}\frac{G_{2^\delta k}(qn_1n_2)}{qn_1n_2}\\
&\quad\times\check{J}_{u+2\gamma}\left(\frac{2^lkX}{16a^2qn_1n_2}\right)V\left(\frac{n_1}{N_1}\right)V\left(\frac{n_2}{N_2}\right)dt_1dt_2dudv,
\end{aligned}
\end{equation*}
where the Fourier-type transform of $J_{u+2\gamma}$ is given by
\begin{equation*}
\check{J}_{u+2\gamma}\left(\frac{2^lkX}{16a^2qn_1n_2}\right)=\int_{-\infty}^{+\infty}(\cos+\sin)\left(\frac{2\pi x\cdot2^lkX}{16a^2qn_1n_2}\right)J(x)x^{u+2\gamma}dx.
\end{equation*}
Since $J_{u+2\gamma}$ is supported on $[0, \infty)$, by Mellin inversion and using (17.43.3, 17.43.4) of \cite{GR} (for detailed process, please refer to \cite[p.1104]{SY}), we have
\begin{equation*}
\check{J}_{u+2\gamma}\left(\frac{2^lkX}{16a^2qn_1n_2}\right)=\frac{1}{2\pi i}\int_{(1/2)}\widetilde{J}(u+2\gamma+1-s)\Gamma(s)(\cos+\sgn(k)\sin)\left(\frac{\pi s}{2}\right)\left(\frac{2\pi kX\cdot2^l}{16a^2qn_1n_2}\right)^{-s}ds.
\end{equation*}
Let
\begin{equation*}
\begin{aligned}
T(N_1, N_2; \theta, t_1, t_2, u, \gamma)&=\sum_{\substack{k\neq0\\(k, 2)=1}}\mathop{\sum\sum}_{(n_1n_2, 2a)=1}
\frac{\lambda_f(n_1)\lambda_f(n_2)}{n_1^{1/2+it_1}n_2^{1/2+it_2}}\frac{G_{2^\delta k}(qn_1n_2)}{qn_1n_2}\\
&\quad\times\check{J}_{u+2\gamma}\left(\frac{kX\theta}{2qn_1n_2}\right)V\left(\frac{n_1}{N_1}\right)V\left(\frac{n_2}{N_2}\right),
\end{aligned}
\end{equation*}
so that we can express $T_{l}$ as
\begin{equation*}
\begin{aligned}
T_{l}&={\sum_{N_1, N_2}}^d\frac{1}{(2\pi i)^4}\int_{(c_1)}\int_{(c_2)}\int_{-\infty}^{+\infty}\int_{-\infty}^{\infty}T(N_1, N_2; 2^l/8a^2, t_1, t_2, u, \gamma)\\
&\quad\times\widetilde{G}(it_1-u-\gamma_1)\widetilde{G}(it_2-v-\gamma_2)N_1^{it_1-u-\gamma_1}N_2^{it_2-v-\gamma_2}\frac{g_{\gamma_1}(u)g_{\gamma_2}(v)}{uv}X^uU^vdt_1dt_2dudv.
\end{aligned}
\end{equation*}
For convenience, we define
\begin{equation*}
\begin{aligned}
&\quad U(N_1, N_2; 2^l/8a^2, t_1, t_2)\\
&:=\max\{|T(N_1, N_2; 2^l/8a^2, t_1, t_2, u, \gamma)|: -4+2/\log X\leq\Re u\leq4-|\Re\gamma|,\,|\gamma|\leq(\log X)^{-1}\}.
\end{aligned}
\end{equation*}
\begin{lemma}\label{lem.6.4}
With notation as above, we have the following inequality
\begin{equation}\label{eq.6.5}
\begin{aligned}
T_{l}&\ll {\sum_{N_1, N_2}}^d\left(1+\frac{N_1}{X}\right)^{-4}\left(1+\frac{N_2}{U}\right)^{-4}\\
&\quad\times\int_{-\infty}^{+\infty}\int_{-\infty}^{+\infty}U(N_1, N_2; 2^l/8a^2, t_1, t_2)\frac{dt_1}{(1+|t_1|)^{10}}\frac{dt_2}{(1+|t_2|)^{10}}.
\end{aligned}
\end{equation}
\end{lemma}
\begin{proof} We divide this sum into four cases as follows.\\
(1) $N_1>X$, $N_2>U$. Moving the contours of the integration to $\Re u=4-\Re\gamma_1$, $\Re v=4-\Re\gamma_2$. Since
\begin{equation*}
\left(\frac{N_1}{X}\right)^{-4}\ll\left(1+\frac{N_1}{X}\right)^{-4},\quad\left(\frac{N_2}{U}\right)^{-4}\ll\left(1+\frac{N_2}{U}\right)^{-4},\quad X^{-\Re\gamma_1},U^{-\Re\gamma_2}\ll1,
\end{equation*}
this case is bounded by
\begin{equation*}
\begin{aligned}
&\ll{\sum_{\substack{N_1>X\\ N_2>U}}}^d
\left(1+\frac{N_1}{X}\right)^{-4}\left(1+\frac{N_2}{U}\right)^{-4}\\
&\quad\times\int_{(4-\Re\gamma_2)}\int_{(4-\Re\gamma_1)}\int_{-\infty}^{+\infty}\int_{-\infty}^{+\infty}U(N_1, N_2; 2^l/8a^2, t_1, t_2)\\
&\quad\frac{dt_1}{(1+|it_1-u-\gamma_1|)^{10}}\frac{dt_2}{(1+|it_2-v-\gamma_2|)^{10}}\frac{du}{(1+|u+\gamma_1|)^{20}}\frac{dv}{(1+|v+\gamma_2|)^{20}}\\
&\ll{\sum_{\substack{N_1>X\\ N_2>U}}}^d
\left(1+\frac{N_1}{X}\right)^{-4}\left(1+\frac{N_2}{U}\right)^{-4}\\
&\quad\times\int_{(4-\Re\gamma_2)}\int_{(4-\Re\gamma_1)}\int_{-\infty}^{+\infty}\int_{-\infty}^{+\infty}U(N_1, N_2; 2^l/8a^2, t_1, t_2)\\
&\quad\times\frac{dt_1}{(1+|t_1|)^{10}}\frac{dt_2}{(1+|t_2|)^{10}}\frac{du}{(1+|u+\gamma_1|)^{10}}\frac{dv}{(1+|v+\gamma_2|)^{10}}.
\end{aligned}
\end{equation*}
(2) $N_1\leq X$, $N_2>U$. Moving the contours of the integration to $\Re u=-4+2/\log X$, $\Re v=4-\Re\gamma_2$. The integrand has pole at $u=0$, and possibly at $u=-1-\gamma_1,-2-\gamma_1, -3-\gamma_1, -4-\gamma_1$ (those poles except for at $u=0$ arise from $g_{\gamma_1}(u)$-term). One may easily verify that these poles give sufficiently small bound so their contributions are negligable. Note that, in this case
\begin{equation*}
\left(\frac{N_1}{X}\right)^{4}\ll\left(1+\frac{N_1}{X}\right)^{-4},\quad\left(\frac{N_2}{U}\right)^{-4}\ll\left(1+\frac{N_2}{U}\right)^{-4},\quad N_1^{-\Re \gamma_1}, U^{-\Re\gamma_2}\ll1,
\end{equation*}
then the integration over integrations to $\Re u=-4+2/\log X$ and $\Re v=4-\Re\gamma_2$ is bounded by
\begin{align*}
&\ll\sum_{\substack{N_1\leq X\\ N_2>U}}
\left(1+\frac{N_1}{X}\right)^{-4}\left(1+\frac{N_2}{U}\right)^{-4}\\
&\quad\times\int_{(4-\Re\gamma_2)}\int_{(-4+2/\log X)}\int_{-\infty}^{+\infty}\int_{-\infty}^{+\infty}U(N_1, N_2; 2^l/8a^2, t_1, t_2)\\
&\quad\times\frac{dt_1}{(1+|it_1-u-\gamma_1|)^{10}}\frac{dt_2}{(1+|it_2-v-\gamma_2|)^{10}}\frac{du}{(1+|u+\gamma_1|)^{20}}\frac{dv}{(1+|v+\gamma_2|)^{20}}\\
&\ll{\sum_{\substack{N_1\leq X\\ N_2>U}}}^d
\left(1+\frac{N_1}{X}\right)^{-4}\left(1+\frac{N_2}{U}\right)^{-4}\\
&\quad\times\int_{(4-\Re\gamma_2)}\int_{(-4+2/\log X)}\int_{-\infty}^{+\infty}\int_{-\infty}^{+\infty}U(N_1, N_2; 2^l/8a^2, t_1, t_2)\\
&\quad\times\frac{dt_1}{(1+|t_1|)^{10}}\frac{dt_2}{(1+|t_2|)^{10}}\frac{du}{(1+|u+\gamma_1|)^{10}}\frac{dv}{(1+|v+\gamma_2|)^{10}}.
\end{align*}
(3) $N_1> X$, $N_2\leq U$. Moving the contours of the integration to $\Re u=4-\Re\gamma_1$, $\Re v=-4+2/\log X$. The integrand has pole at $v=0$, and possibly at $v=-1-\gamma_2,-2-\gamma_2, -3-\gamma_2, -4-\gamma_2$ (those poles except for at $v=0$ arise from $g_{\gamma_2}(v)$-term). One may easily verify that these poles give sufficiently small bound so their contributions are negligable. Note that, in this case
\begin{equation*}
\left(\frac{N_1}{X}\right)^{-4}\ll\left(1+\frac{N_1}{X}\right)^{-4},\quad\left(\frac{N_2}{U}\right)^{4}\ll\left(1+\frac{N_2}{U}\right)^{-4},\quad X^{-\Re\gamma_1},N_2^{-\Re\gamma_2}\ll1
\end{equation*}
then the integration over integrations to $\Re u=4-\Re\gamma_1$ and $\Re v=-4+2/\log X$ is bounded by
\begin{equation*}
\begin{aligned}
&\ll\sum_{\substack{N_1> X\\ N_2\leq U}}
\left(1+\frac{N_1}{X}\right)^{-4}\left(1+\frac{N_2}{U}\right)^{-4}\\
&\quad\times\int_{(-4+2/\log X)}\int_{(4-\Re\gamma_1)}\int_{-\infty}^{+\infty}\int_{-\infty}^{+\infty}U(N_1, N_2; 2^l/8a^2, t_1, t_2)\\
&\quad\times\frac{dt_1}{(1+|it_1-u-\gamma_1|)^{10}}\frac{dt_2}{(1+|it_2-v-\gamma_2|)^{10}}\frac{du}{(1+|u+\gamma_1|)^{20}}\frac{dv}{(1+|v+\gamma_2|)^{20}}\\
&\ll{\sum_{\substack{N_1>X\\ N_2\leq U}}}^d
\left(1+\frac{N_1}{X}\right)^{-4}\left(1+\frac{N_2}{U}\right)^{-4}\\
&\quad\times\int_{(-4+2/\log X)}\int_{(4-\Re\gamma_1)}\int_{-\infty}^{+\infty}\int_{-\infty}^{+\infty}U(N_1, N_2; 2^l/8a^2, t_1, t_2)\\
&\quad\frac{dt_1}{(1+|t_1|)^{10}}\frac{dt_2}{(1+|t_2|)^{10}}\frac{du}{(1+|u+\gamma_1|)^{10}}\frac{dv}{(1+|v+\gamma_2|)^{10}}.
\end{aligned}
\end{equation*}
(4) $N_1\leq X$, $N_2\leq U$. Moving the contours of the integration to $\Re u=-4+2/\log X$, $\Re v=-4+2/\log X$. The integrand has pole at $u=0$, $v=0$, and possibly at $u=-1-\gamma_1,-2-\gamma_1, -3-\gamma_1, -4-\gamma_1$, $v=-1-\gamma_2,-2-\gamma_2, -3-\gamma_2, -4-\gamma_2$ (those poles except for at $u=0$, $v=0$ arise from $g_{\gamma_1}(u)$, $g_{\gamma_2}(v)$-term). These poles give sufficiently small bound so their contributions are negligable.
Note that
\begin{equation*}
\left(\frac{N_1}{X}\right)^{4}\ll1\ll\left(1+\frac{N_1}{X}\right)^{-4},\quad\left(\frac{N_2}{U}\right)^{4}\ll1\ll\left(1+\frac{N_2}{U}\right)^{-4},\quad N_1^{-\Re \gamma_1},N_2^{-\Re \gamma_2}\ll1,
\end{equation*}
in a similar way, the sum over $N_1\leq X$, $N_2\leq U$ is
\begin{equation*}
\begin{aligned}
&\ll{\sum_{\substack{N_1\leq X\\ N_2\leq U}}}^d
\left(1+\frac{N_1}{X}\right)^{-4}\left(1+\frac{N_2}{U}\right)^{-4}\\
&\quad\times\int_{(-4+2/\log X)}\int_{(-4+2/\log X)}\int_{-\infty}^{+\infty}\int_{-\infty}^{+\infty}U(N_1, N_2; 2^l/8a^2, t_1, t_2)\\
&\quad\times\frac{dt_1}{(1+|t_1|)^{10}}\frac{dt_2}{(1+|t_2|)^{10}}\frac{du}{(1+|u+\gamma_1|)^{10}}\frac{dv}{(1+|v+\gamma_2|)^{10}}.
\end{aligned}
\end{equation*}
Therefore, from the above, we can obtain formula \eqref{eq.6.5}.
\end{proof}
The estimation of $T(N_1, N_2; \theta, t_1, t_2, u, \gamma)$ is shown by the following Lemma \ref{lem.TN1N2}.
\begin{lemma}\label{lem.TN1N2}
With notation as above, for $-4+2/\log X\leq\Re u\leq4-|\Re\gamma|$, we have
\begin{equation*}
T(N_1, N_2; \theta, t_1, t_2, u, \gamma)\ll\frac{qa^\varepsilon\sqrt{N_1N_2}}{\theta X}(1+|t_1|)^{2}(1+|t_2|)^{2}.
\end{equation*}
\end{lemma}
\begin{proof}
We apply the Mellin inversion to $G(n_1/N_1)$, $G(n_2/N_2)$, then $T(N_1, N_2; \theta, t_1, t_2, u, \gamma)$ is a finite sum of terms of the form
\begin{equation*}
\begin{aligned}
\mathbb{T}&=\frac{1}{(2\pi i)^3}\int_{(1/2)}\int_{(2)}\int_{(2)}\sum_{\substack{k \neq 0\\(k, 2)=1}}\left(\frac{2q}{2\pi\theta Xk}\right)^s\widetilde{J}(u+2\gamma+1-s) \Gamma(s)\left(\cos+\sgn(k)\sin\right)\left(\frac{\pi s}{2}\right)\\
&\quad\times\mathop{\sum\sum}_{\substack{n_1, n_2 \\ \left(n_1 n_2, 2a\right)=1}}\frac{\lambda_f\left(n_1\right) \lambda_f\left(n_2\right)}{n_1^{1/2-s+it_1+w_1} n_2^{1/2-s+it_2+w_2}}
\frac{G_{2^\delta k}\left(qn_1 n_2\right)}{qn_1 n_2}\widetilde{G}(w_1)\widetilde{G}(w_2)N_{1}^{w_1}N_{2}^{w_2}dw_1dw_2ds.
\end{aligned}
\end{equation*}
Further, we write $2^\delta k=k_1k_{2}^{2}$, where $k_1$ is squarefree and $(k_2, 2)=1$, $k_2$ is positive. Let $K=8qN_1N_2/\theta X$. Then
\begin{equation*}
\begin{aligned}
\mathbb{T}&=\frac{1}{(2\pi i)^3}\int_{(1/2)}\int_{(2)}\int_{(2)}{\sum_{|k_1|\geq1}}^*\left(\frac{2q}{2^{1-\delta}\pi\theta Xk_1}\right)^s\widetilde{J}(u+2\gamma+1-s) \Gamma(s)
\left(\cos+\sgn(k_1)\sin\right)\left(\frac{\pi s}{2}\right)\\
&\quad\times\widetilde{G}(w_1)\widetilde{G}(w_2)N_{1}^{w_1}N_{2}^{w_2}H(1/2-s+it_1+w_1, 1/2-s+it_2+w_2, s; k_1, a)dw_1dw_2ds\\
&=\frac{1}{(2\pi i)^3}\int_{(1/2)}\int_{(2)}\int_{(2)}\left({\sum_{|k_1|\leq K}}^*+{\sum_{|k_1|> K}}^*\right)\left(\frac{2q}{2^{1-\delta}\pi\theta Xk_1}\right)^s\\
&\quad\times\widetilde{J}(u+2\gamma+1-s) \Gamma(s)\left(\cos+\sgn(k_1)\sin\right)\left(\frac{\pi s}{2}\right)\widetilde{G}(w_1)\widetilde{G}(w_2)N_{1}^{w_1}N_{2}^{w_2}\\
&\quad\times H(1/2-s+it_1+w_1, 1/2-s+it_2+w_2, s; k_1, a)dw_1dw_2ds
\end{aligned}
\end{equation*}
where ${\sum}^*$ denotes the sum over squarefree integers and
\begin{equation*}
H(\rho_1, \rho_2, \rho_3; k_1, a):=\sum_{\substack{k_2\geq1\\(k_2, 2)=1}}\mathop{\sum\sum}_{\substack{n_1, n_2 \\ \left(n_1 n_2, 2a\right)=1}}\frac{\lambda_f\left(n_1\right) \lambda_f\left(n_2\right)}{n_1^{\rho_1} n_2^{\rho_2}k_{2}^{2\rho_3}}\frac{G_{k_1k_{2}^{2}}\left(qn_1 n_2\right)}{qn_1 n_2}.
\end{equation*}
It can be known from Lemma 2.5 in \cite{Li} that
\begin{equation*}
H(\rho_1, \rho_2, \rho_3; k_1, a)=L\left(1/2+\rho_1, f \otimes \chi_m\right) L\left(1/2+\rho_2, f \otimes \chi_m\right)Y\left(\rho_1, \rho_2, \rho_3; k_1\right),
\end{equation*}
for
\begin{equation*}
\begin{aligned}
&\quad Y\left(\rho_1, \rho_2, \rho_3; k_1\right)\\
&=\frac{H_2(\rho_1, \rho_2, \rho_3)}{\zeta(1+\rho_1+\rho_2) L\left(1+2\rho_1, \operatorname{sym}^2 f\right)L\left(1+\rho_1+\rho_2, \operatorname{sym}^2 f\right) L\left(1+2\rho_2, \operatorname{sym}^2 f\right)}
\end{aligned}
\end{equation*}
where $m=k_1$ if $k_1\equiv1(\bmod 4)$ and $m=4k_1$ if $k_1\equiv2, 3(\bmod 4)$, $H_2(\rho_1, \rho_2, \rho_3)$ is analytic in the region $\Re \rho_1, \Re \rho_2\geq-\delta_0/2$ and $\Re\rho_3\geq1/2+\delta_0$ for any $0<\delta_0<1/3$. Moreover, in the same region, $H_2(\rho_1, \rho_2, \rho_3)\ll \tau(a)$ where the implied constant may depend on $\delta_0$ and $f$.\par
We may write a multiple Dirichlet series for $Y$ of the form
\begin{equation*}
Y\left(\rho_1, \rho_2, \rho_3; k_1\right)=\sum_{r_1, r_2, r_3}\frac{C(r_1, r_2, r_3)}{r_{1}^{\rho_1}r_{2}^{\rho_2}r_{3}^{2\rho_3}}.
\end{equation*}
It follows from the inverse Mellin transform that
\begin{equation*}
\begin{aligned}
&\frac{1}{(2\pi i)^2}\int_{(2)}\int_{(2)}\widetilde{G}(w_1)\widetilde{G}(w_2)N_{1}^{w_1}N_{2}^{w_2}H\left(\tfrac{1}{2}-s+it_1+w_1, \tfrac{1}{2}-s+it_2+w_2, s; k_1, a\right)dw_1dw_2\\
&=\frac{1}{(2\pi i)^2}\int_{(2)}\int_{(2)}\sum_{r_1, r_2, r_3}\frac{C(r_1, r_2, r_3)}{r_{1}^{\frac{1}{2}+it_1-s}r_{2}^{\frac{1}{2}+it_2-s}r_{3}^{2s}}\mathop{\sum\sum}_{n_1, n_2}\frac{\lambda_f(n_1)\chi_{m}(n_1)}{n_{1}^{1+it_1-s}}\frac{\lambda_f(n_2)\chi_{m}(n_2)}{n_{2}^{1+it_2-s}}\\
&\quad\times\widetilde{G}(w_1)\widetilde{G}(w_2)\left(\frac{r_1n_1}{N_1}\right)^{-w_1}\left(\frac{r_2n_2}{N_2}\right)^{-w_2}dw_1dw_2\\
&=\sum_{r_1, r_2, r_3}\frac{C(r_1, r_2, r_3)}{r_{1}^{\frac{1}{2}+it_1-s}r_{2}^{\frac{1}{2}+it_2-s}r_{3}^{2s}}\mathop{\sum\sum}_{n_1, n_2}\frac{\lambda_f(n_1)\chi_{m}(n_1)}{n_{1}^{1+it_1-s}}\frac{\lambda_f(n_2)\chi_{m}(n_2)}{n_{2}^{1+it_2-s}}G\left(\frac{r_1n_1}{N_1}\right)G\left(\frac{r_2n_2}{N_2}\right).
\end{aligned}
\end{equation*}
We apply a partition of unity to the sum over $r_1$, $r_2$. Then the above formula is equal to
\begin{equation*}
\begin{aligned}
&{\sum_{R_1, R_2}}^d\sum_{r_1, r_2, r_3}\frac{C(r_1, r_2, r_3)}{r_{1}^{\frac{1}{2}+it_1-s}r_{2}^{\frac{1}{2}+it_2-s}r_{3}^{2s}}G\left(\frac{r_1}{R_1}\right)G\left(\frac{r_2}{R_2}\right)\\
&\times\mathop{\sum\sum}_{n_1, n_2}
\frac{\lambda_f(n_1)\chi_{m}(n_1)}{n_{1}^{1+it_1-s}}\frac{\lambda_f(n_2)\chi_{m}(n_2)}{n_{2}^{1+it_2-s}}G\left(\frac{r_1n_1}{N_1}\right)G\left(\frac{r_2n_2}{N_2}\right)V\left(\frac{R_1n_1}{N_1}\right)V\left(\frac{R_2n_2}{N_2}\right)
\end{aligned}
\end{equation*}
where ${\sum_{R_1, R_2}}^d$ denotes a sum over $R_{i}=2^{t_i}$ for integer $t_i\geq0$. Since $G\left(\frac{r_in_i}{N_i}\right)$ restricts $\frac{3N_i}{4r_{i}}\leq n_i\leq\frac{2N_i}{r_i}$ while $G\left(\frac{r_i}{R_i}\right)$ restricts $\frac{3R_i}{4}\leq r_i\leq 2R_i$, the above holds for any $V$ which is identically 1 on $[3/8, 8/3]$. We set
\begin{equation*}
V(x)=G(4x)+G(2x)+G(x)+G(x/2)+G(x/4)
\end{equation*}
where $G$ is our fixed function satisfying (\ref{eq.V}) so that $V$ is identically $1$ on $[1/4, 6]$.\par
We once again apply Mellin inversion to separate variables inside $G\left(\frac{r_1n_1}{N_1}\right)G\left(\frac{r_2n_2}{N_2}\right)$, then
\begin{equation*}
\begin{aligned}
\mathbb{T}&={\sum_{R_1, R_2}}^d\frac{1}{(2\pi i)^3}\int_{(1/2)}\int_{(2)}\int_{(2)}\left({\sum_{|k_1|\leq K}}^*+{\sum_{|k_1|> K}}^*\right)\left(\frac{2q}{2^{1-\delta}\pi\theta Xk_1}\right)^s
\Gamma(s)\left(\cos+\sgn(k_1)\sin\right)\left(\frac{\pi s}{2}\right) \\
&\quad\times\widetilde{J}(u+2\gamma+1-s)\sum_{r_1, r_2, r_3}\frac{C(r_1, r_2, r_3)}
{r_{1}^{\frac{1}{2}+it_1-s+w_1}r_{2}^{\frac{1}{2}+it_2-s+w_2}r_{3}^{2s}}G\left(\frac{r_1}{R_1}\right)G\left(\frac{r_2}{R_2}\right)\widetilde{G}(w_1)\widetilde{G}(w_2)\\
&\quad\times\mathop{\sum\sum}_{n_1, n_2}
\frac{\lambda_f(n_1)\chi_{m}(n_1)}{n_{1}^{1+it_1-s+w_1}}\frac{\lambda_f(n_2)\chi_{m}(n_2)}{n_{2}^{1+it_2-s+w_2}}V\left(\frac{R_1n_1}{N_1}\right)V\left(\frac{R_2n_2}{N_2}\right)N_{1}^{w_1}N_{2}^{w_2}dw_1dw_2ds.
\end{aligned}
\end{equation*}
Changing $-s+w_1$ to $w_1$ and $-s+w_2$ to $w_2$. We have that $\mathbb{T}$ is a finite sum of terms of the form
\begin{equation}\label{eq.5.14}
\begin{aligned}
&{\sum_{R_1, R_2}}^d\frac{1}{(2\pi i)^3}\int_{(\frac{1}{2})}\int_{(\frac{3}{2})}\int_{(\frac{3}{2})}\left({\sum_{|k_1|\leq K}}^*+{\sum_{|k_1|> K}}^*\right)
\left(\frac{2q}{2^{1-\delta}\pi\theta Xk_1}\right)^s \Gamma(s)\left(\cos+\sgn(k_1)\sin\right)\left(\frac{\pi s}{2}\right)\\
&\quad\times\widetilde{J}(u+2\gamma+1-s)\sum_{r_1, r_2, r_3}\frac{C(r_1, r_2, r_3)}
{r_{1}^{\frac{1}{2}+it_1+w_1}r_{2}^{\frac{1}{2}+it_2+w_2}r_{3}^{2s}}G\left(\frac{r_1}{R_1}\right)G\left(\frac{r_2}{R_2}\right)\widetilde{G}(w_1+s)\widetilde{G}(w_2+s)\\
&\quad\times\mathop{\sum\sum}_{n_1, n_2}
\frac{\lambda_f(n_1)\chi_{m}(n_1)}{n_{1}^{1+it_1+w_1}}\frac{\lambda_f(n_2)\chi_{m}(n_2)}{n_{2}^{1+it_2+w_2}}G\left(\frac{n_1}{M_1}\right)G\left(\frac{n_2}{M_2}\right)N_{1}^{w_1+s}N_{2}^{w_2+s}dw_1dw_2ds
\end{aligned}
\end{equation}
where $M_i\asymp N_i/R_i$.

Suppose $\Re s\geq3/5$ and write $w_1=-1/2+iw_{1}'$, $w_2=-1/2+iw_{2}'$ for real $w_{1}'$, $w_{2}'$. It is known from \cite[Lemma 5.7]{Li} that
\begin{equation}\label{eq,5.7}
\begin{aligned}
&\left|\sum_{r_1, r_2, r_3}\frac{C(r_1, r_2, r_3)}
{r_{1}^{\frac{1}{2}+it_1+w_1}r_{2}^{\frac{1}{2}+it_2+w_2}r_{3}^{2s}}G\left(\frac{r_1}{R_1}\right)G\left(\frac{r_2}{R_2}\right)\right|\\
&\ll a^{\varepsilon}(1+|t_1|)^{1/10}(1+|t_2|)^{1/10}(1+|w_1'|)(1+|w_2'|)\exp(-c_1\sqrt{\log(R_{1}R_{2})}),
\end{aligned}
\end{equation}
where the implied constant and $c_1>0$ may depend on $f$.

$T(N_1, N_2; \alpha, t_1, t_2, u, \gamma)$ is a finite sum of terms of the form (\ref{eq.5.14}). When $|k_1|\leq K$, moving the contour of integration in $s$, $w_1$, $w_2$ to $\Re s=3/5$, $\Re w_1=-1/2$, $\Re w_2=-1/2$. When $|k_1|> K$, moving the contour of integration in $s$, $w_1$, $w_2$ to $\Re s=6/5$, $\Re w_1=-1/2$, $\Re w_2=-1/2$. We only give the proof for $|k_1|\leq K$, the case $|k_1|> K$ is analogous to the case $|k_1|\leq K$.\par
Using that $\Gamma(s)\left(\cos+\sgn(k_1)\sin\right)\left(\frac{\pi s}{2}\right)\ll|s|^{\Re s-1/2}$ and (\ref{eq,5.7}) we have that the $|k_1|\leq K$ part of (\ref{eq.5.14}) is
\begin{equation}\label{eq.5.20}
\begin{aligned}
&\ll a^{\varepsilon}\left(\frac{q}{2^{1-\delta}\pi\theta X}\right)^{\frac{3}{5}}(N_1N_2)^{\frac{1}{10}}(1+|t_1|)^{\frac{1}{10}}(1+|t_2|)^{\frac{1}{10}}{\sum_{R_1, R_2}}^d\exp(-c_1\sqrt{\log (R_{1}R_{2})})\\
&\quad\times\int_{-\infty}^{+\infty}\int_{-\infty}^{+\infty}\frac{1}{(1+|w_1'|)^{10}(1+|w_2'|)^{10}}{\sum_{|k_1|\leq K}}^*\frac{1}{|k_1|^{3/5}}\\
&\quad\times\left|\sum_{n_1}\frac{\lambda_f(n_1)\chi_{m}(n_1)}{n_{1}^{\frac{1}{2}+it_1+iw_1'}}G\left(\frac{n_1}{M_1}\right)\sum_{n_2}\frac{\lambda_f(n_2)\chi_{m}(n_2)}{n_{2}^{\frac{1}{2}+it_2+iw_2'}}G\left(\frac{n_2}{M_2}\right)\right|dw_1'dw_2'.
\end{aligned}
\end{equation}
We split the sum in $k_1$ into dyadic intervals of the form $\mathcal{K}_1\leq|k_1|<2\mathcal{K}_1$, by Lemma \ref{lem.SMNt} and Cauchy--Schwarz inequality, we have
\begin{equation*}
\begin{aligned}
&\quad{\sum_{\mathcal{K}_1\leq|k_1|<2\mathcal{K}_1}}^*\left|\sum_{n_1}\frac{\lambda_f(n_1)\chi_{m}(n_1)}{n_{1}^{\frac{1}{2}+it_1+iw_1'}}G\left(\frac{n_1}{M_1}\right)\right|
\left|\sum_{n_2}\frac{\lambda_f(n_2)\chi_{m}(n_2)}{n_{2}^{\frac{1}{2}+it_2+iw_2'}}G\left(\frac{n_2}{M_2}\right)\right|\\
&\leq{\sum_{\mathcal{K}_1\leq|m|<8\mathcal{K}_1}}^\flat\left|\sum_{n_1}\frac{\lambda_f(n_1)\chi_{m}(n_1)}{n_{1}^{\frac{1}{2}+it_1+iw_1'}}G\left(\frac{n_1}{M_1}\right)\right|
\left|\sum_{n_2}\frac{\lambda_f(n_2)\chi_{m}(n_2)}{n_{2}^{\frac{1}{2}+it_2+iw_2'}}G\left(\frac{n_2}{M_2}\right)\right|\\
&\ll(1+|t_1+w_1'|)^{\frac{3}{2}}(1+|t_2+w_2'|)^{\frac{3}{2}}\mathcal{K}_1\log^{1/2}(2+|t_1+w_1'|)\log^{1/2}(2+|t_2+w_2'|)\\
&\ll\mathcal{K}_1(1+|t_1|)^{\frac{16}{10}}(1+|t_2|)^{\frac{16}{10}}(1+|w_1'|)^2(1+|w_2'|)^2.
\end{aligned}
\end{equation*}
An application of the Abel summation formula, together with the above estimate, yields the bound
\begin{equation*}
\begin{aligned}
&{\sum_{|k_1|\leq K}}^*\frac{1}{|k_1|^{\frac{3}{5}}}\left|\sum_{n_1}\frac{\lambda_f(n_1)\chi_{m}(n_1)}{n_{1}^{\frac{1}{2}+it_1+iw_1'}}G\left(\frac{n_1}{M_1}\right)\right|
\left|\sum_{n_2}\frac{\lambda_f(n_2)\chi_{m}(n_2)}{n_{2}^{\frac{1}{2}+it_2+iw_2'}}G\left(\frac{n_2}{M_2}\right)\right|\\
&\ll K^{\frac{2}{5}}(1+|t_1|)^{\frac{16}{10}}(1+|t_2|)^{\frac{16}{10}}(1+|w_1'|)^2(1+|w_2'|)^2.
\end{aligned}
\end{equation*}
Substituting this estimate into \eqref{eq.5.20}, we obtain a bound of
\begin{equation*}
\ll\frac{qa^\varepsilon\sqrt{N_1N_2}}{\theta X}(1+|t_1|)^{2}(1+|t_2|)^{2},
\end{equation*}
where we have estimated ${\sum_{R_1, R_2}}^d\exp(-c_1\sqrt{\log (R_{1}R_{2})})=\sum_{h_1, h_2\geq0}\exp(-c_1\sqrt{\log 2}\sqrt{h_1+h_2})=\sum_{h\geq0}(h+1)\exp(-c_1\sqrt{\log 2}\sqrt{h})\ll 1$. This completes the proof of Lemma \ref{lem.TN1N2}.
\end{proof}
According to Lemma \ref{lem.6.4} and Lemma \ref{lem.TN1N2} ($\theta=2^{l}/8a^2$), we can obtain
\begin{equation*}
T_l\ll\frac{qa^{2+\varepsilon}}{2^lX}(UX)^{1/2}.
\end{equation*}
Recall that $U=X/(\log X)^{1000}$ and $Y_1=Y_2=(\log X)^{20}$, then
\begin{align*}
&\quad\mathcal{N}(\beta, -\alpha, 8d)(k\neq0),\;\mathcal{N}(-\beta, \alpha, 8d)(k\neq0)\\
&\ll X\sum_{a\leq Y_1}\frac{1}{a^2}\sum_{l\geq0}|T_l|\ll X\sum_{a\leq Y_1}\frac{qa^{\varepsilon}}{X}(UX)^{1/2}
\ll_{q}Y_{1}^{1+\varepsilon}(UX)^{1/2}\ll_{q}X.
\end{align*}
In the same way ($\theta=2^{l}/8a^2b$), it follows that
\begin{equation*}
\begin{aligned}
&\quad\mathcal{N}(\beta, \alpha, 8d)(k\neq0),\; \mathcal{N}(-\beta, -\alpha, 8d)(k\neq0)\\
&\ll X\sum_{a\leq Y_1}\frac{1}{a^2}\sum_{\substack{b\mid q\\ b\leq Y_2}}\frac{1}{b}\frac{a^{2+\varepsilon}b}{X}(UX)^{1/2}\ll_{q}Y_{1}^{1+\varepsilon}(UX)^{1/2}\ll_{q}X.
\end{aligned}
\end{equation*}
Therefore, the off-diagonal part can be absorbed by $O_q(X)$. Theorem \ref{thm3.1} follows from combining Proposition \ref{Pro3.2} and Proposition \ref{Pro3.3}.

\section{The proofs of Theorem 1.2 and Theorem 1.3}
According to Theorem \ref{thm3.1}, we have the following result. This result is crucial to the proofs of Theorem \ref{thm1} and Theorem \ref{thm2}.
\begin{theorem}\label{thm4.1}
Let $J(x): (0, \infty)\rightarrow\mathbb{C}$ be a smooth function compactly supported in $[1/2, 2]$. For any normalized primitive form $f$ with odd level $q$ and weight $2$,
any $l_1, l_2\in\mathbb{Z}_{\geq0}$, we have
\begin{equation*}
\begin{aligned}
&{\sum_{(d, 2q)=1}}^*L^{(l_1)}\left(\tfrac{1}{2}, f\otimes\chi_{8d}\right)L^{(l_2)}\left(\tfrac{1}{2}, f\otimes\chi_{8d}\right)J\left(\frac{8d}{X}\right)\\
&=\mathcal{C}_{l_1+l_2+1}X(\log X)^{l_1+l_2+1}+O_{q, l_1, l_2}(X(\log\log X)^{4}(\log X)^{l_1+l_2}),
\end{aligned}
\end{equation*}
the leading coefficient $\mathcal{C}_{l_1+l_2+1}$ is
\begin{equation}\label{eq.xishu}
\begin{aligned}
&\quad\frac{(-2)^{l_1+l_2+1}}{2\pi^2}\frac{-1}{l_1+l_2+1}\widetilde{J}(1)L^3(1, \operatorname{sym}^2 f)Z_1^*\left(\tfrac{1}{2}, \tfrac{1}{2}\right)
+\widetilde{J}(1)L^3(1, \operatorname{sym}^2 f)(-\eta)Z_q^*\left(\tfrac{1}{2}, \tfrac{1}{2}\right)\\
&\times\frac{(-2)^{l_1+l_2+1}}{2\pi^2}\left\{\frac{(-1)^{l_2+1}-1}{2}+\frac{(-1)^{l_2}-1}{2}\right\}B(l_1+1, l_2+1)
\end{aligned}
\end{equation}
and the big-$O$ is depending on $J$, $q$, $l_1$, $l_2$. Here, $\widetilde{J}(\omega)=\int_{0}^{+\infty}J(x)x^{\omega-1}dx$ is the Mellin transform of $J$, $Z_{q'}^*(1/2, 1/2)$ is defined in (\ref{eq.z_q'}), $B(l_1+1, l_2+1)$ is the beta function.
\end{theorem}
\begin{proof}
By \eqref{eq.W} and (\ref{eq.z_q'}), the main terms in (\ref{eq.LL}) contain the apparent singularities $\alpha\pm\beta=0$. Since the $L$-function is entire, these singularities are necessarily cancelled.
Put
\begin{equation}\label{eq.A}
\sum_{\epsilon_1, \epsilon_2\in\{\pm1\}}
\gamma_{\alpha}^{\delta_1}\gamma_{\beta}^{\delta_2}X^{-2\delta_1\alpha-2\delta_2\beta}\widetilde{J}(1-2\delta_1\alpha-2\delta_2\beta)
Z_{q'}(\epsilon_1\alpha, \epsilon_2\beta)=\frac{A(\alpha, \beta; J)}{(\alpha+\beta)(\alpha-\beta)}.
\end{equation}
Then, since the left side of (\ref{eq.A}) is entire in a neighborhood of $\alpha=\pm\beta$, $A(\alpha, \beta; J)$ is analytic in a neighborhood of $\alpha=\pm\beta$ and divisible by $(\alpha+\beta)(\alpha-\beta)$.
Assume that the function $A(\alpha, \beta; J)=\sum_{m, n\geq0}c_{m, n}\alpha^m\beta^n$ is divisible by $(\alpha+\beta)(\alpha-\beta)$.
Then $\frac{\partial^{l_1+l_2}}{\partial\alpha^{l_1}\partial\beta^{l_2}}\frac{A(\alpha, \beta; J)}{(\alpha+\beta)(\alpha-\beta)}|_{\alpha=\beta=0}$ is given by
\begin{equation*}
\sum_{n=0}^{l_2}\sum_{k=0}^{n}\sum_{m=0}^{l_1}\left(\begin{matrix}l_2\\n\end{matrix}\right)\left(\begin{matrix}n\\k\end{matrix}\right)\left(\begin{matrix}l_1\\m\end{matrix}\right)
(-1)^{l_1+l_2-n-m}\frac{(n-k)!(l_2-n)!(l_1+l_2+2-k)!k!}{(l_2-k+1)!(l_1+l_2+2-k-m)}c_{l_1+l_2+2-k, k}.
\end{equation*}
The proof process refers to the Lemma 2.3 in \cite{Son}. Hence, we have
{\small\begin{equation}\label{eq.2.10}
\begin{aligned}
&\quad\frac{\partial^{l_1+l_2}}{\partial\alpha^{l_1}\partial\beta^{l_2}}\left\{\sum_{\epsilon_1, \epsilon_2\in\{\pm1\}}
\gamma_{\alpha}^{\delta_1}\gamma_{\beta}^{\delta_2}X^{-2\delta_1\alpha-2\delta_2\beta}\widetilde{J}(1-2\delta_1\alpha-2\delta_2\beta)
Z_{q'}(\epsilon_1\alpha, \epsilon_2\beta)\right\}_{\alpha=\beta=0}\\
&=\sum_{n=0}^{l_2}\sum_{k=0}^{n}\sum_{m=0}^{l_1}\left(\begin{matrix}l_2\\n\end{matrix}\right)\left(\begin{matrix}n\\k\end{matrix}\right)\left(\begin{matrix}l_1\\m\end{matrix}\right)
(-1)^{l_1+l_2-n-m}\frac{(n-k)!(l_2-n)!(l_1+l_2+2-k)!k!}{(l_2-k+1)!(l_1+l_2+2-k-m)}\\
&\quad\times\frac{1}{(l_1+l_2+2-k)!k!}\frac{\partial^{l_1+l_2+2}}{\partial\alpha^{l_1+l_2+2-k}\partial\beta^k}A(\alpha, \beta; J)\lvert_{\alpha=\beta=0}\\
&=\sum_{n=0}^{l_2}\sum_{k=0}^{n}\sum_{m=0}^{l_1}\left(\begin{matrix}l_2\\n\end{matrix}\right)\left(\begin{matrix}n\\k\end{matrix}\right)\left(\begin{matrix}l_1\\m\end{matrix}\right)
(-1)^{l_1+l_2-n-m}\frac{(n-k)!(l_2-n)!}{(l_2-k+1)!(l_1+l_2+2-k-m)}\\
&\quad\times\frac{\partial^{l_1+l_2+2}}{\partial\alpha^{l_1+l_2+2-k}\partial\beta^k}\left\{(\alpha^2-\beta^2)\sum_{\epsilon_1, \epsilon_2\in\{\pm1\}}
\gamma_{\alpha}^{\delta_1}\gamma_{\beta}^{\delta_2}X^{-2\delta_1\alpha-2\delta_2\beta}\widetilde{J}(1-2\delta_1\alpha-2\delta_2\beta)
Z_{q'}(\epsilon_1\alpha, \epsilon_2\beta)\right\}_{\alpha=\beta=0}.
\end{aligned}
\end{equation}}
We see that $\frac{\partial^{i+j}}{\partial\alpha^i\partial\beta^j}X^{-2\delta_1\alpha-2\delta_2\beta}$ is asymptotically equal a polynomial of $\log X$ of degree $i+j$. Further, it should be mentioned that when we apply the Leibniz rule to compute the derivatives in the right hand side of (\ref{eq.2.10}), it follows that
\begin{equation*}
\left[(\alpha^2-\beta^2)\sum_{\epsilon_1, \epsilon_2\in\{\pm1\}}
\gamma_{\alpha}^{\delta_1}\gamma_{\beta}^{\delta_2}\widetilde{J}(1-2\delta_1\alpha-2\delta_2\beta)
Z_{q'}(\epsilon_1\alpha, \epsilon_2\beta)\frac{\partial^{l_1+l_2+2}}{\partial\alpha^{l_1+l_2+2-k}\partial\beta^k}X^{-2\delta_1\alpha-2\delta_2\beta}\right]_{\alpha=\beta=0}=0,
\end{equation*}
because $(\alpha^2-\beta^2)Z_{q'}(\epsilon_1\alpha, \epsilon_2\beta)$ is divisible by at least one of $\alpha+\beta$ and $\alpha-\beta$. Similarly, the $W_{q', \epsilon_2\beta}$-term in (\ref{eq.LL}) can be handled in the same way. Hence by (\ref{eq.LL}), (\ref{eq.2.10}) and the above reason, it follows that
\begin{equation*}
\begin{aligned}
&{\sum_{(d, 2q)=1}}^*L^{(l_1)}\left(\tfrac{1}{2}, f\otimes\chi_{8d}\right)L^{(l_2)}\left(\tfrac{1}{2}, f\otimes\chi_{8d}\right)J\left(\frac{8d}{X}\right)\\
&=\mathcal{C}_{l_1+l_2+1}X(\log X)^{l_1+l_2+1}+O_{q, l_1, l_2}(X(\log\log X)^{4}(\log X)^{l_1+l_2}),
\end{aligned}
\end{equation*}
where $\mathcal{C}_{l_1+l_2+1}$ is a constant which is dependent of $J$. We can calculate $\mathcal{C}_{l_1+l_2+1}$ as shown in formula (\ref{eq.xishu}).
Note that the error term in Theorem \ref{thm3.1} is holomorphic on the disc centred at $(0, 0)$ with radius $\ll\frac{1}{\log X}$. Hence the size of the derivative of the error term is $O_{q, l_1, l_2}(X(\log\log X)^{4}(\log X)^{l_1+l_2})$ by Cauchy's integral formula.
\end{proof}
We prove Theorem \ref{thm1} and Theorem \ref{thm2} by Theorem \ref{thm4.1}. Theorem \ref{thm2} is a direct consequence of Theorem \ref{thm4.1}. For any integer $i\geq0$, if $f$ is a normalized primitive form of odd level $q$ and weight $2$, then Theorem \ref{thm4.1} and the work of Shen\cite{SQ} directly yield
\begin{equation}\label{eq.3.2}
\#\{d \leq X | \text{$d$ is a fundamental discriminant with $(d, q)=1$, }L^{(i)}(1/2, f\otimes\chi_d)\neq0\}\gg\frac{X}{\log X}.
\end{equation}
We now give the proof of Theorem \ref{thm1}.
\begin{proof} By the work of Shen\cite{SQ}, we have
\begin{equation}\label{EQ.L}
\begin{aligned}
&\quad{\sum_{(d, 2q)=1}}^*L\left(\tfrac{1}{2}+\alpha, f\otimes\chi_{8d}\right)J\left(\frac{8d}{X}\right)\\
&=\frac{X}{2\pi^2}\widetilde{J}(1)L(1+2\alpha, \operatorname{sym}^2 f)Z_1\left(\alpha\right)\\
&\quad+(-\eta)\frac{\gamma_\alpha X^{1-2\alpha}}{2\pi^2}\widetilde{J}(1-2\alpha)L(1-2\alpha, \operatorname{sym}^2 f)Z_q\left(-\alpha\right)+O_{q}(X^{\frac{1}{2}+\varepsilon}),
\end{aligned}
\end{equation}
where $q'=1, q$, the big-$O$ is depending on $\varepsilon$, $q$ and $J$. The symbol $\gamma_\alpha$ is shown in Theorem \ref{thm3.1}, and $Z_{q'}\left(\pm\alpha\right)$ is as follow
\begin{equation*}
L(1\pm2\alpha, \operatorname{sym}^2 f)Z_{q'}\left(\pm\alpha\right)
=\sum_{\substack{(n, 2)=1\\ q'n=\square}}\frac{\lambda_f(n)}{n^{\frac{1}{2}\pm\alpha}}\prod_{p\mid qn}\frac{p}{p+1},
\end{equation*}
where $Z_{q'}(\pm\alpha)$ is analytic and absolutely convergent in the region $\Re(\pm\alpha)\geq-1/4+\varepsilon$.
By taking the $i$-th derivative with respect to $\alpha$ and then letting $\alpha\rightarrow0$, we obtian
\begin{equation}\label{eq.1.5}
\begin{aligned}
&\quad{\sum_{(d, 2q)=1}}^*L^{(i)}\left(\tfrac{1}{2}, f\otimes\chi_{8d}\right)J\left(\frac{8d}{X}\right)\\
&=\frac{(-2)^i\times(-\eta)}{2\pi^2}\widetilde{J}(1)L(1, \operatorname{sym}^2 f)Z_{q}(0)X\log^i X+\sum_{j=0}^{i-1}C_{q}^jX(\log X)^j+O_{i, q}(X^{\frac{1}{2}+\varepsilon})
\end{aligned}
\end{equation}
for some constants $C_{q}^j$. Let $l_1=l_2=i$, according to Theorem \ref{thm4.1}, we have
\begin{equation}\label{eq.1.4}
{\sum_{(d, 2q)=1}}^*L^{(i)}\left(\tfrac{1}{2}, f\otimes\chi_{8d}\right)^2J\left(\frac{8d}{X}\right)=\mathcal{C}_{2i+1}X(\log X)^{2i+1}+O_{i}(X(\log\log X)^{4}(\log X)^{2i}).
\end{equation}
A similar method works for all discriminants $d$ as well.
If $f$ corresponds to an elliptic curve $E$, applying Cauchy--Schwarz inequality to (\ref{eq.1.5}) and (\ref{eq.1.4}), we have
\begin{equation*}
\left(\sum_{d\leq X}\left|L^{(i)}\left(E^{(d)}, 1\right)\right|\right)^2\leq\sum_{d\leq X}\left|L^{(i)}\left(E^{(d)}, 1\right)\right|^2
\sum_{\substack{d\leq X\\L^{(i)}\left(E^{(d)}, 1\right)\neq0}}1,
\end{equation*}
thus the proof of Theorem \ref{thm1} is completed.
\end{proof}

\bigskip

\noindent Tong Wei, {\it Research Center for Mathematics and Interdisciplinary Sciences, Shandong University, Qingdao, Shandong, China.}

\smallskip
\noindent {\it E-mail:} 202421344@mail.sdu.edu.cn

\bigskip

\noindent Shuai Zhai, {\it Mathematical Research Center, Shandong University, Jinan, Shandong, China.}

\smallskip
\noindent {\it E-mail:} zhai@sdu.edu.cn

\end{document}